\pdfoutput=1
\documentclass[11pt]{article}
\usepackage{graphicx} 
\usepackage{diagbox}
\usepackage{authblk}
\makeatletter
\renewcommand\AB@authnote[1]{\textsuperscript{#1}\hspace{5pt}}

\makeatother
\usepackage{hyperref}
\usepackage{algorithm}
\usepackage{algpseudocode}
\usepackage{notation-2}
\usepackage{arxiv-2}
\usepackage{amssymb}
\usepackage{mathrsfs}
\usepackage{float}
\usepackage{setspace}
\usepackage{lmodern}        

\renewcommand{\tilde}[1]{\widetilde{#1}}

\renewcommand{\hat}[1]{\widehat{#1}}

\newcommand{\pll}{\kern 0.3em/\kern -0.9em /\kern 0.3em}

\usepackage{microtype}
\title{\normalfont Simultaneous Pointwise Majorization for Mixed Tail Processes with Applications in Gaussian Chaos and Ergodic Diffusions
}
\begin{centering}
     \author[1]{Haichen Hu\thanks{{Email: \texttt{huhc@mit.edu}}}}
      \author[1]{David Simchi-Levi\thanks{{Email: \texttt{dslevi@mit.edu}}}}
\end{centering}
\affil[1]{{\small\itshape
Massachusetts Institute of Technology, 77 Massachusetts Avenue,
Cambridge, MA 02139, USA}}
\begin{document}
\maketitle
\singlespacing
\vspace{-2em}
\begin{abstract}
Classical chaining controls an indexed stochastic process through a
single worst-case bound and can therefore obscure substantial variation
across the index set. We develop the first simultaneous pointwise
majorization theory for Banach-valued processes with
finite-metric mixed-tail increments. Suppose that an anchored process
$(Z_t)_{t\in T}$ satisfies, for some integer $m\ge1$, pseudo-metrics
$d_1,\ldots,d_m$, and orders $\alpha_1,\ldots,\alpha_m>0$,

\begin{align*}
\PP\{
\|Z_t-Z_s\|>
\sum_{j=1}^m u^{1/\alpha_j}d_j(s,t)
\}
\le 2e^{-u},
s,t\in T.
\end{align*}

For ambient priors $\mu_1,\ldots,\mu_m$, let $v_j(t)
:=d_j(t,t_0), \Phi_j(t):=
\int_0^{4v_j(t)}
(
\log\frac{1}{\mu_j(B_{d_j}(t,r))}
)^{1/\alpha_j}dr$.
We prove that, $\forall \delta\in(0,1)$, with probability at least
$1-\delta$, simultaneously for all $t\in T$,
\begin{align*}
\|Z_t\|
\le
C_{m,\boldsymbol\alpha}
\sum_{j=1}^m
\{
\Phi_j(t)
+
v_j(t)
(\log(e/\delta))^{1/\alpha_j}
\}.
\end{align*}
Here $\boldsymbol\alpha:=(\alpha_1,\ldots,\alpha_m)$ and
$C_{m,\boldsymbol\alpha}$ depend only on $m$ and these tail orders.
The result subsumes single-metric sub-Weibull processes of every positive order as the case $m=1$. In the Gaussian setting, it sharpens the pointwise upper bound of \citet{xu2026} by eliminating the logarithmic terms generated by dyadic peeling. The proof retains the index-wise costs
of measure-generated admissible chains and synchronizes the regimes through a nested common refinement. Finally, we  apply our theorems to stationary diffusion empirical processes and decoupled Gaussian chaos to obtain simultaneous pointwise envelope bounds, which can further be applied to other statistics problems.
\end{abstract}
\vspace{-2em}

\section{Introduction}
\label{sec:introduction}

Uniform control of stochastic processes is a basic tool in statistics
and applied probability. It underlies the analysis of empirical risk
minimization, random matrices and structured embeddings, and statistical
procedures based on continuously observed stochastic systems. In many
of these problems, however, the index set is highly heterogeneous: some
indices belong to locally simple regions, whereas others lie in much
more complex parts of the parameter space. A single worst-case bound for
the entire process does not reflect this variation.

Generic chaining estimates the supremum of an indexed stochastic process
from the tail behavior of its increments and the metric complexity of
its index set. It grew out of the classical entropy method, the
Garsia--Rodemich--Rumsey inequality, and the majorizing-measure method
\citep{dudley1967sizes,garsia1971real,fernique1971regularite,
preston1972continuity,fernique1975regularite}. Talagrand's majorizing
measure theorem and the subsequent generic-chaining theory identify the
correct global complexity of Gaussian processes and provide a general
framework for controlling stochastic suprema
\citep{talagrand1987regularity,talagrand1992simple,
talagrand1994constructions,talagrand1996majorizing,
talagrand2001without,talagrand2005generic}. This framework has also led
to algorithmic constructions of majorizing measures and to deviation
bounds for processes with general Orlicz and mixed-tail increments
\citep{talagrand1990sample,borst2021optimizer,dirksen2015tail}. Its
standard output is nevertheless global: every index is controlled by
the same upper bound for the supremum.

Such a global output can be unnecessarily conservative when an index is
chosen after the process has been observed. A fixed-index concentration
inequality preserves the scale of that index but cannot simply be
substituted at a random, data-dependent choice. A supremum bound remains
valid after selection but charges the selected index with the most
complex part of the full parameter space. We seek both properties at
once: for every $\delta\in(0,1)$, we construct a deterministic function
$\mathfrak M_\delta:T\to[0,\infty]$ such that
\begin{align}
\PP\cbr{
\|Z_t\|\le \mathfrak M_\delta(t)
\text{ for every }t\in T
}
\ge 1-\delta,
\label{eq:intro_simultaneous_pointwise_majorization}
\end{align}
while $\mathfrak M_\delta(t)$ depends on the local metric complexity and
fluctuation scales at that same index. We call
\eqref{eq:intro_simultaneous_pointwise_majorization} a
\emph{simultaneous pointwise majorization}. Because the event is common
to all indices, the bound may be evaluated at a measurable index chosen
from the same realization.

We develop the first simultaneous pointwise majorization theory for
Banach-valued processes with finite-metric mixed-tail
increments. Let
$(Z_t)_{t\in T}$ be a Banach-valued process anchored at $t_0$, and
suppose that, for pseudo-metrics $d_1,\ldots,d_m$, tail orders
$\alpha_1,\ldots,\alpha_m>0$, and every $u\ge0$,
\begin{align}
\PP\cbr{
\|Z_t-Z_s\|>
\sum_{j=1}^m u^{1/\alpha_j}d_j(s,t)
}
\le 2e^{-u}.
\label{eq:intro_finite_regime_increment}
\end{align}
The same deviation parameter controls every term in
\eqref{eq:intro_finite_regime_increment}. Thus, the assumption neither
imposes $m$ independent single-metric inequalities nor presupposes a
decomposition of the process into $m$ components.

For every regime, we fix a Borel probability measure $\mu_j$ before
observing the process and define
\begin{align}
v_j(t):=d_j(t,t_0),\ 
\Phi_j(t):=
\int_0^{4v_j(t)}
\left(
\log\frac{1}{\mu_j(B_{d_j}(t,r))}
\right)^{1/\alpha_j}dr.
\label{eq:intro_finite_regime_functionals}
\end{align}
The integral in \eqref{eq:intro_finite_regime_functionals} is a
pointwise Fernique–Talagrand functional \citep{xu2026}. It measures the mass available
around $t$ at successive scales and is truncated at the distance from
$t$ to the anchor, rather than at a global radius of the index space.
Our main result shows that, with probability at least $1-\delta$,
simultaneously for every $t\in T$,
\begin{align}
\|Z_t\|
\le
C_{m,\boldsymbol\alpha}
\sum_{j=1}^m
\left[
\Phi_j(t)
+
v_j(t)
\left(\log\frac{e}{\delta}\right)^{1/\alpha_j}
\right].
\label{eq:intro_finite_regime_bound}
\end{align}
Here $\boldsymbol\alpha:=(\alpha_1,\ldots,\alpha_m)$, and
$C_{m,\boldsymbol\alpha}$ depends only on $m$ and these tail orders.
The conclusion is uniform in probability but pointwise in the
deterministic threshold. In particular, each regime preserves its own
pseudo-metric, prior measure, local ball-mass term, anchor radius,
and tail order.

When $m=1$, \eqref{eq:intro_finite_regime_bound} yields a simultaneous
pointwise envelope for every sub-Weibull process of arbitrary order
$\alpha>0$ and subsumes the finite Gaussian upper envelope of
\citet{xu2026} as a special case. As in that work, the bound is governed
by a fixed-measure ball-mass integral truncated at the individual anchor
radius. Even in the Gaussian setting, however, our bound is strictly
sharper: it eliminates the logarithmic loss introduced by dyadic peeling.
This improvement follows from constructing a measure-generated admissible
chain whose deterministic cost is retained at each index, rather than
replaced by a worst-case value over a radius or complexity shell. The same
construction also extends to arbitrary positive sub-Weibull
orders and separable index spaces.

When $m=2$, the result gives a pointwise counterpart of the familiar
mixed-tail generic-chaining bound. The two increment scales may have
different pseudo-metrics and different tail orders, but both act on the
same process increment through one deviation parameter. Applying the
single-metric result twice would therefore be invalid. Our proof instead
constructs a measure-generated chain for every pseudo-metric and embeds
their histories in a nested common refinement. The same idea extends to
any fixed finite number of regimes. A deterministic level shift keeps
the common refinement admissible and makes the dependence of
$C_{m,\boldsymbol\alpha}$ on the number of regimes and the tail orders
explicit. No failure probability is allocated over pointwise complexity
or anchor-radius shells.

Other related approaches include minorizing-metric and chaining bounds
for sample-path regularity, stochastic suprema, and concentration
\citep{kwapien2004sample,bednorz2006theorem,bednorz2007holder,
bednorz2014concentration}; local and problem-dependent complexity methods
in learning theory
\citep{bartlett2005local,koltchinskii2006local,xu2025towards}; and
PAC--Bayesian or multiscale coding formulations
\citep{audibert2007combining,maurer2010codes,chu2026codes}. These lines
of work address all-pairs increments, global suprema, shared random
quantities, or learning-theoretic risks. To our knowledge, they do not
provide, for finite-metric mixed-tail processes, the particular output
proved here: a single event on which each index receives an additive
fixed-measure ball-mass term and an anchor-distance confidence term for
every increment regime.

We provide two applications. The first concerns the occupation empirical
process of a stationary scalar diffusion. Such processes unify the empirical
distribution function, invariant-density estimators, and contrast
functions used in $M$-estimation
\citep{kutoyants1997nonparametric,vanzanten2003empirical,
vandervaart2005donsker,kutoyants2010goodness}. By invoking a Poincar\'e-based
Bernstein inequality \citep{gao2014bernstein,aeckerle2021concentration}, our result for two-metric mixed tail processes
turns these fixed-observable increments into one finite-horizon event
that controls the entire observable class while retaining the local
complexity and fluctuation scales of each observable.

The second application concerns higher-order decoupled Gaussian chaos,
which appears in Gaussian polynomial inequalities, multiple stochastic
integrals, canonical $U$-statistics, and random tensors
\citep{delapena1999decoupling,major2014multiple,bamberger2022hanson}. We apply Lata{\l}a's moment inequality \citep{latala2006estimates} and prove that decoupled Gaussian chaos satisfies our finite-metric mixed-tail property associated with the partition norms, thereby providing a simultaneous tensor-specific envelope for decoupled Gaussian chaos.

\paragraph{Paper Structure.} The remainder of the paper is organized as follows. As a starter, Section~2 develops
the theorem regarding two-metric mixed tail processes and records the corollary for sub-Weibull processes as a
specialization. Then, Section~3 establishes the theorem for general finite-metric mixed tail stochastic processes, which acts as our main result.
In Sections~4 and~5, we discuss the application of our theorems in diffusion empirical processes and the decoupled
Gaussian chaos process, respectively. The appendices collect the external mathematical tools and all
proofs.

\section{Pointwise Majorization for Two-Metric Mixed Tail Processes}\label{sec:mixed_tail-majorization}
In this section, we establish a pointwise majorization theorem for stochastic processes whose
increments satisfy a two-metric mixed tail condition. The two regimes may have
different tail exponents and may be governed by different pseudo-metrics.

Let $(\mathbb B,\|\cdot\|)$ be a separable Banach space, and let
$(T,\mathcal A)$ be a standard Borel space. Let
$d_1,d_2:T\times T\longrightarrow[0,\infty)$ be finite-valued, jointly
$\mathcal A\otimes\mathcal A$-measurable pseudo-metrics. Define the pseudo-metric $\rho$ as
\begin{align*}
\rho(s,t):=d_1(s,t)+d_2(s,t),
\qquad s,t\in T,
\end{align*}
and assume that $(T,\rho)$ is separable. All random objects below are
defined on a fixed probability space $(\Omega,\mathcal H,\PP)$.

Let $Z:\Omega\times T\longrightarrow\mathbb B$ be jointly measurable
from $(\Omega\times T,\mathcal H\otimes\mathcal A)$ to
$(\mathbb B,\mathcal B(\mathbb B))$, and write
\begin{align*}
Z_t(\omega):=Z(\omega,t),
\ \omega\in\Omega,\ t\in T.
\end{align*}
Now, we introduce some standard notation in studying empirical processes. First, we assume that $\cbr{Z_t}_{t\in T}$ has a continuous modification such that the map $t\mapsto Z_t(\omega)$ is continuous from $(T,d)$ to $(\mathbb B,\|\cdot\|)$. With a little abuse of notation, we still denote it by
$(Z_t)_{t\in T}$. Second, we assume that there exists an anchor $t_0\in T$ such that $Z_{t_0}=0$ almost surely.

To summarize, there exists $t_0\in T$ and an event
$\Omega_0\in\mathcal H$ with $\PP(\Omega_0)=1$ such that, for every
$\omega\in\Omega_0$, $Z_{t_0}(\omega)=0$
, the map $t\mapsto Z_t(\omega)$ 
is continuous from $(T,d)$ to $(\mathbb B,\|\cdot\|)$. Throughout this paper, we will work conditioned on the event $\Omega_0$ and thus we omit its dependence.

To study the suprema of our sub-Weibull process, we recall some background about generic chaining.
For the pseudo-metric space $(T,d)$, a sequence $\cT=(T_n)_{n\ge 0}$ of subsets of $T$ is \emph{admissible} if $|T_0|=1$, $|T_n|\le 2^{2^n}$. The $\alpha$-Talagrand functional \citep{dirksen2015tail} is defined as
\[
\gamma_{\alpha}(T,d):=\inf_{(T_n)}\sup_{t\in T}\sum_{n\ge 0}2^{n/\alpha}d(t,T_n),\ d(t,T_n)=\inf_{s\in T_n}d(t,s).
\]
The diameter of any set $S\subset T$ with respect to $d$ is defined as $\text{diam}_d(S)=\sup_{s,t\in S}d(t,s)$. Moreover, we use $B_d(t,r)$ to denote the local neighborhood centered at $t$ with radius $r$, i.e., $B_d(t,r):=\cbr{s:d(t,s)\le r}$. We can also define the $\alpha$-Talagrand functional on the subset $S$ as $$\gamma_{\alpha}(S,d):=\inf_{(S_n)}\sup_{t\in S}\sum_{n\ge 0}2^{n/\alpha}d(t,S_n),\ d(t,S_n)=\inf_{s\in S_n}d(t,s),$$
where $S_n$ is an admissible sequence of subsets of $S$.

Now we return to the discussion of $\cbr{Z_t}_{t\in T}$.  $(Z_t)_{t\in T}$ is said to be a two-metric mixed-tail process of
orders $\alpha_1,\alpha_2>0$ with respect to $(d_1,d_2)$ if, for every
$s,t\in T$ and every $u\ge0$, we have that
\begin{align}
\PP\cbr{
\|Z_t-Z_s\|>
u^{1/\alpha_1}d_1(t,s)
+u^{1/\alpha_2}d_2(t,s)
}
\le2e^{-u}.
\label{eq:mixed_tail_increment_condition}
\end{align}
The same parameter $u$ controls the two terms in
\eqref{eq:mixed_tail_increment_condition}, but each term retains its own
tail exponent and pseudo-metric.

The distinctive feature of
\eqref{eq:mixed_tail_increment_condition} is that one deviation
parameter activates two fluctuation regimes on the same probability
event. The two terms are therefore not separate tail bounds that are
assumed independently. In particular, neither $d_1$ nor $d_2$ is
required by itself to control the increments. When
$(\alpha_1,\alpha_2)=(1,2)$, the threshold takes the familiar
Bernstein form
\begin{align*}
u d_1(s,t)+\sqrt{u}\,d_2(s,t).
\end{align*}
In this case, $d_2$ describes the variance-driven sub-Gaussian regime,
whereas $d_1$ describes the sub-exponential correction that becomes
decisive for larger deviations. More general pairs
$(\alpha_1,\alpha_2)$ allow the two geometries to have different
sub-Weibull orders.

This structure occurs whenever an increment inequality involves two
genuinely different scale parameters. For a centered empirical process
based on $m$ independent observations, a Bernstein inequality typically
produces a variance metric of the form
$m^{-1/2}\|f_t-f_s\|_{L^2(P)}$ and an envelope metric of the form
$m^{-1}\|f_t-f_s\|_\infty$; this is the standard mixed
sub--Gaussian–sub-exponential setting \citep{dirksen2015tail}. For matrix-indexed centered quadratic forms of
a vector with independent centered sub-Gaussian coordinates, the
Hanson–Wright inequality yields a Frobenius-norm metric in the
sub-Gaussian term and an operator-norm metric in the sub-exponential
term \citep{rudelson2013hanson}. Centered additive functionals of
symmetric or reversible Markov processes provide another example when
a Poincar\'e or related functional inequality yields a finite-time
Bernstein form \citep{gao2014bernstein}. The two
applications developed later in this paper: quadratic chaos and
stationary scalar diffusions, are concrete instances of the latter
two mechanisms.

Fix probability measures $\mu_1$ and $\mu_2$ on
$(T,\mathcal A)$. For
$j\in\cbr{1,2}$, we define the pointwise anchor radius $v_j(t):=d_j(t,t_0)$
and the pointwise Fernique–Talagrand functionals of order $\alpha_1, \alpha_2$ \citep{xu2026} as
\begin{align*}
\Phi_{\mu_j,d_j}^{(\alpha_j)}(t)
:=
\int_0^{4v_j(t)}
\left(
\log\frac{1}{\mu_j(B_{d_j}(t,r))}
\right)^{1/\alpha_j}dr,
\end{align*}
where $B_{d_j}(t,r)
:=
\cbr{s\in T:d_j(s,t)\le r}$. The integrals are understood in the extended-real sense, with
$\log(1/0):=+\infty$. For notational simplicity, throughout this
section and its appendix, we write
\begin{align*}
\Phi_j(t)
:=
\Phi_{\mu_j,d_j}^{(\alpha_j)}(t),
\ j\in\cbr{1,2}.
\end{align*}

We now state the main theorem of this section.

\begin{theorem}
\label{thm:mixed_tail_majorization}
Under the preceding setup, let $\cbr{ Z_t}_{t\in T}$ be an anchored
two-metric mixed-tail process of orders $\alpha_1,\alpha_2>0$. There
exists a finite constant $C_{\alpha_1,\alpha_2}$, depending only on
$\alpha_1$ and $\alpha_2$, such that, for every $\delta\in(0,1)$, with
probability at least $1-\delta$, simultaneously for every $t\in T$, we
have
\begin{align*}
\|Z_t\|
\le
C_{\alpha_1,\alpha_2}
\cbr{
\Phi_1(t)+\Phi_2(t)
+v_1(t)
\left(
\log\frac{e}{\delta}
\right)^{1/\alpha_1}
+v_2(t)
\left(
\log\frac{e}{\delta}
\right)^{1/\alpha_2}
}.
\end{align*}
If either $\Phi_1(t)$ or $\Phi_2(t)$ is infinite, the corresponding
inequality holds trivially.
\end{theorem}

The conclusion of Theorem \ref{thm:mixed_tail_majorization} is
pointwise in its deterministic envelope but simultaneous in its
probability event. More precisely, there exists a single event of
probability at least $1-\delta$ on which the displayed inequality holds
for every $t\in T$, while the right-hand side depends on the local
geometric quantities associated with the individual index $t$.

The first two terms $\Phi_1(t), \Phi_2(t)$ describe the local complexity and measure the local multiscale complexity around $t$. At each scale $r$,
the quantity $\mu(B_{d_i}(t,r))$ records the mass assigned by the ambient prior measure $\mu_i$ to the neighborhood of $t$. Thus, larger local ball masses give a smaller integrand and hence a smaller pointwise
complexity. Moreover, the integral is taken only up to $4v_i(t)$, so the complexity profiles involve only scales comparable to the distance from $t$ to the anchor.

The remaining terms are the pointwise confidence penalty. Since
$Z_{t_0}=0$ almost surely and $v_i(t)=d_i(t,t_0)$, the increment
condition gives, for every $u\ge0$, we have
\begin{align*}
\PP\cbr{
\|Z_t\|>u^{1/\alpha_1}v_1(t)+u^{1/\alpha_2}v_2(t)
}=
\PP\cbr{
\|Z_t-Z_{t_0}\|>u^{1/\alpha_1}v_1(t)+u^{1/\alpha_2}v_2(t)
}\le2e^{-u}.
\end{align*}
Therefore, $v_1(t),v_2(t)$ are the natural fluctuation scale at the index $t$.
Theorem \ref{thm:mixed_tail_majorization} preserves this pointwise
scale in the simultaneous bound: the confidence penalty is proportional
to $\sum_{i=1}^2v_i(t)
\left(
\log\frac{e}{\delta}
\right)^{1/\alpha_i}$, rather than to the global diameter of $T$.

An important feature of the theorem is that no dyadic-shell terms appear in the conclusion. In the proof, we retain the pointwise cost of the
measure-generated admissible chain. Thus, we do not first replace this
cost by a supremum over a subset and then distribute the failure
probability among dyadic shells. This is why no shell-allocation
logarithms appear.

The proof of Theorem \ref{thm:mixed_tail_majorization} relies on a
series of lemmas, and the full proof is deferred to
Appendix \ref{app:proofs_mixed_tail}.

For the frequently occurring Bernstein tail condition with ordering
$(\alpha_1,\alpha_2)=(1,2)$, Theorem \ref{thm:mixed_tail_majorization} implies
\begin{align*}
\|Z_t\|
\le
C_{1,2}
\cbr{
\Phi_1(t)+\Phi_2(t)
+v_1(t)\log\frac{e}{\delta}
+v_2(t)\sqrt{\log\frac{e}{\delta}}
}
\end{align*}
simultaneously for every $t\in T$. If one of the pseudo-metrics is
identically zero, its profile and anchor radius vanish, and the result
reduces to the following corollary.
\begin{corollary}\label{cor:sub-weibull}
Let $\cbr{Z_t}_{t\in T}$ be an anchored
sub-Weibull process of order $\alpha>0$, i.e., 
\begin{align*}
\PP\cbr{
\|Z_t-Z_s\|>u^{1/\alpha}d(t,s)
}
\le 2e^{-u}.
\end{align*}
Fix a probability
measure $\mu$ on $(T,\mathcal A)$. There exists a finite constant $C_\alpha$,
depending only on $\alpha$, such that, for every $\delta\in(0,1)$, with
probability at least $1-\delta$, simultaneously for every $t\in T$, we
have
\begin{align*}
\|Z_t\|
\le
C_\alpha
\left[
\Phi_{\mu,d}^{(\alpha)}(t)
+
v(t)
\left(
\log\frac{e}{\delta}
\right)^{1/\alpha}
\right].
\end{align*}
\end{corollary}
For $\alpha=2$, Corollary \ref{cor:sub-weibull} gives the
sub-Gaussian bound $\|Z_t\|
\le
C_2
\left[
\Phi_{\mu,d}^{(2)}(t)
+
v(t)\sqrt{\log\frac{e}{\delta}}
\right]$
simultaneously for every $t\in T$, which improves the previous results in \citet{xu2026} by removing the auxiliary peeling logarithms from the upper envelope. For $\alpha=1$, the
corresponding sub-exponential confidence penalty is $v(t)\log\frac{e}{\delta}$.

\section{Pointwise Majorization for Finite-Metric Mixed-Tail Processes}
\label{sec:multi_regime_majorization}

In this section, we generalize our pointwise majorization theorem to
stochastic processes whose increments are governed by any finitely many tail
regimes. The single-metric sub-Weibull theorem and the two-metric
mixed-tail theorem are recovered as the cases in which the number of
regimes is one and two, respectively.

We first give a brief review to the problem setup. Set $(\mathbb B,\|\cdot\|)$ to be a separable Banach space and let
$(T,\mathcal A)$ be a standard Borel space, and fix an integer $m\ge1$.
For every $j\in\cbr{1,\ldots,m}$, let
$d_j:T\times T\longrightarrow[0,\infty)$ be a finite-valued, jointly
$\mathcal A\otimes\mathcal A$-measurable pseudo-metric, and let
$\alpha_j>0$. We define
\begin{align*}
\rho(s,t)
:=
\sum_{j=1}^m d_j(s,t),
\ s,t\in T,
\end{align*}
and assume that $(T,\rho)$ is separable. All random objects below are
defined on a fixed probability space $(\Omega,\mathcal H,\PP)$.

Smilar to Section \ref{sec:mixed_tail-majorization}, for a nonempty subset $S\subseteq T$, we write $d_j(t,S)
:=
\inf_{s\in S}d_j(t,s)$.
Similarly, a sequence $\mathcal T=(T_n)_{n\ge0}$ of nonempty subsets of $T$ is called
\emph{admissible} if $|T_0|=1,\ |T_n|\le2^{2^n}$.

Let $Z:\Omega\times T\longrightarrow\mathbb B$ be jointly measurable
from $(\Omega\times T,\mathcal H\otimes\mathcal A)$ to
$(\mathbb B,\mathcal B(\mathbb B))$, and write
\begin{align*}
Z_t(\omega):=Z(\omega,t),
\ \omega\in\Omega,\ t\in T.
\end{align*}
We assume that the process has a modification, still denoted by
$(Z_t)_{t\in T}$, for which
there exist an anchor $t_0\in T$ and an event
$\Omega_0\in\mathcal H$ with $\PP(\Omega_0)=1$ such that, for every
$\omega\in\Omega_0$, we have $Z_{t_0}(\omega)=0$
and the map $t\mapsto Z_t(\omega)$ is continuous from $(T,\rho)$ to
$(\mathbb B,\|\cdot\|)$. We work with this modification throughout the section.

The measure-generated admissible sequences used below are anchored at
$T_0=\cbr{t_0}$.

We say that $(Z_t)_{t\in T}$ is an \emph{$m$-regime mixed-tail process}
of orders $\boldsymbol\alpha=(\alpha_1,\ldots,\alpha_m)$ with respect to
$(d_1,\ldots,d_m)$ if, for every $s,t\in T$ and every $u\ge0$,
\begin{align}
\PP\cbr{
\|Z_t-Z_s\|
>
\sum_{j=1}^m u^{1/\alpha_j}d_j(t,s)
}
\le2e^{-u}.
\label{eq:multi_regime_increment_condition}
\end{align}
The same deviation parameter $u$ controls all the terms in
\eqref{eq:multi_regime_increment_condition}. Thus, the assumption is not
equivalent to imposing $m$ separate single-metric increment inequalities.
In particular, no individual pseudo-metric $d_j$ is required by itself to
control the increments, and we do not assume that the process decomposes
as a sum of $m$ processes.

The different terms in \eqref{eq:multi_regime_increment_condition} may
become dominant at different deviation levels. For instance, if
$\alpha_j=2/j$, then the $j$th term is proportional to $u^{j/2}$.
This sequence of powers is the one that appears in the moment and tail
estimates for decoupled Gaussian chaoses. When $m=2$ and
$(\alpha_1,\alpha_2)=(1,2)$, the threshold reduces to the familiar
Bernstein tail condition $u\,d_1(s,t)+\sqrt{u}\,d_2(s,t)$. When $m=1$, condition \eqref{eq:multi_regime_increment_condition} is the
usual sub-Weibull increment condition of order $\alpha_1$.

For every $j\in\cbr{1,\ldots,m}$, fix a Borel probability measure
$\mu_j$ on $(T,\mathcal A)$ before observing the process. We define the
pointwise anchor radius
\begin{align*}
v_j(t):=d_j(t,t_0)
\end{align*}
and the pointwise
Fernique-Talagrand functional
\begin{align}
\Phi_j(t)
:=
\int_0^{4v_j(t)}
\left(
\log\frac{1}{\mu_j(B_{d_j}(t,r))}
\right)^{1/\alpha_j}dr,
\label{eq:multi_regime_pointwise_functional}
\end{align}
where $B_{d_j}(t,r)
:=
\cbr{s\in T:d_j(s,t)\le r}$.
The measure, the neighborhood, and the power in
\eqref{eq:multi_regime_pointwise_functional} are allowed to depend on
the regime. The integral is taken only over scales comparable with the
$d_j$-distance from the individual index $t$ to the anchor.

We now state the main theorem about our $m$-regime mixed tail process.
\begin{theorem}[Multi-regime pointwise majorization]
\label{thm:multi_regime_majorization}
Under the preceding setup, let $(Z_t)_{t\in T}$ be an anchored
$m$-regime mixed-tail process of orders
$\boldsymbol\alpha=(\alpha_1,\ldots,\alpha_m)$. There exists a finite
constant
\begin{align*}
C_{m,\boldsymbol\alpha}
=
C_{m,\alpha_1,\ldots,\alpha_m}<\infty,
\end{align*}
depending only on $m$ and the tail orders, such that, for every
$\delta\in(0,1)$, with probability at least $1-\delta$,
simultaneously for every $t\in T$, we have
\begin{align}
\|Z_t\|
\le
C_{m,\boldsymbol\alpha}
\sum_{j=1}^m
\left[
\Phi_j(t)
+
v_j(t)
\left(
\log\frac{e}{\delta}
\right)^{1/\alpha_j}
\right].
\label{eq:multi_regime_pointwise_bound}
\end{align}
\end{theorem}

The conclusion is simultaneous in its probability event but pointwise in
its deterministic envelope. More precisely, one event controls every
$t\in T$, while the right-hand side of
\eqref{eq:multi_regime_pointwise_bound} retains the local ball masses and
the anchor radius of that particular index in every one of the $m$
pseudo-metrics. The theorem does not replace these quantities by the
corresponding global diameters or by a worst-case complexity over the full
index space.

The dependence of the constant on $m$ is explicit in the proof. Let
\begin{align*}
b_m
:=
1+\ceil{\log_2m}.
\end{align*}
At every level $n\ge b_m$, the common partition records all
metric-specific nearest-net labels only through level $n-b_m$. Since
\begin{align*}
m\sum_{k=0}^{n-b_m}2^k
\le
2^n,
\end{align*}
the resulting common refinement remains admissible. The corresponding
finite-set chaining estimate contains the deterministic factor
$\max_j2^{b_m/\alpha_j}$, together with constants depending only on the
tail orders. We retain this dependence inside
$C_{m,\boldsymbol\alpha}$ rather than claiming that the constant is
uniform in the number of regimes.

An important feature of the theorem is that no failure probability is
allocated over pointwise complexity or radius shells. For every metric,
the fixed measure $\mu_j$ generates an admissible chain whose cost is
retained at each individual index. The chains are then synchronized by a
nested common refinement. This deterministic refinement makes all chains
telescope along the same process increments, while the probabilistic
argument still uses one event. Consequently, the final
confidence term contains only the common confidence parameter and the
regime-specific anchor radius.

The proof of Theorem \ref{thm:multi_regime_majorization}, including the
finite-set common-refinement construction and the passage to the full
separable index space, is deferred to
Appendix \ref{app:proofs_multi_regime}.

\section{Pointwise Envelopes for Diffusion Empirical Processes}
\label{sec:pointwise_ergodic_diffusion}

In this section, we give our first application of the pointwise
majorization theorem. We first introduce a stationary scalar diffusion and
the occupation empirical process indexed by a class of observables. We then
identify its two increment geometries and apply Theorem
\ref{thm:mixed_tail_majorization}.

Let $W$ be a standard Brownian motion and consider
\begin{align}
dX_s
&=
b(X_s)ds+\sigma(X_s)dW_s,
\ s\ge0.
\label{eq:diffusion_sde}
\end{align}
We also use $a(x)$ to denote the square of the function $\sigma(x)$, i.e., $a(x):=\sigma^2(x)$.
\begin{assumption}
\label{ass:diffusion_coefficients}
The functions $b,\sigma:\RR\to\RR$ are globally Lipschitz. There exist
$0<\underline\sigma\le\overline\sigma<\infty$, $A>0$, and $\gamma>0$
such that $\underline\sigma
\le |\sigma(x)|
\le \overline\sigma,
\ \text{if}\ x\in\RR$, and that $\frac{b(x)}{a(x)}\operatorname{sgn}(x)
\le-\gamma,
\ \text{if}\ |x|\ge A$.
\end{assumption}
We define
$G(x):=
\int_0^x\frac{2b(y)}{a(y)}dy,\ 
\mathsf Z:=
\int_{\RR}\frac{e^{G(x)}}{a(x)}dx,\ 
\pi(x)
:=
\frac{e^{G(x)}}{\mathsf Z a(x)},\ 
\ell(x):=
a(x)\pi(x)
=
\frac{e^{G(x)}}{\mathsf Z}$.

Global Lipschitz continuity implies the local Lipschitz and linear-growth
conditions in \citet{aeckerle2021concentration}. Therefore, we know that there is a unique global strong solution with ergodic properties,
$0<\mathsf Z<\infty$, and the density $\pi$ defined above is invariant. We start the diffusion
in equilibrium, $X_0\sim\pi$, independently of $W$, and write $\PP_\pi$
and $\EE_\pi$ for the resulting stationary law.

Now, we specify the diffusion empirical process that we are going to study. For any Borel bounded function class $\cF$ and fixed observation horizon $\tau>0$, given the realized diffusion process $\cbr{X_s}_{0\le s\le \tau}$, we denote the empirical occupation measure as $\widehat\pi_\tau:=
\frac1\tau\int_0^\tau\delta_{X_s}ds$.

Then, for any $f\in\cF$, we denote
\begin{align*}
\widehat\pi_\tau(f):=
\frac1\tau\int_0^\tau f(X_s)ds,\ 
\pi(f):=
\int_{\RR}f(x)\pi(x)dx,\ 
\overline f:=f-\pi(f).
\end{align*}
Intuitively, $\widehat\pi_\tau(f)$ is the sample path average of $f$ over the trajectory $\cbr{X_s}_{0\le s\le \tau}$. $\pi(f)$ is the expectation of $f$ under the invariant probability measure, and $\overline f$ is its centralized version.

Then, our \emph{diffusion empirical process} $\cbr{Z_{\tau}(f)}_{f\in\cF}$ indexed by $\cF$ is defined as
\begin{align*}
Z_\tau(f)
&:=
\sqrt\tau\rbr{\widehat\pi_\tau(f)-\pi(f)}
=
\frac1{\sqrt\tau}\int_0^\tau\overline f(X_s)ds,\ f\in\cF.
\end{align*}
Throughout this section, $\tau$ is fixed, and we
apply pointwise majorization to the process $\cbr{Z_\tau(f)}_{f\in\cF}$ indexed by
the observable $f$. In particular, we have that $Z_\tau(0)=0,
\
Z_\tau(f)-Z_\tau(h)=Z_\tau(f-h)$.

Diffusion empirical processes provide a common framework for statistical
procedures based on continuous observations of an ergodic diffusion.
Indexing the process by indicator functions yields the empirical
distribution function, kernel translates yield estimators of the invariant
density, and classes of contrast functions arise in $M$-estimation
\citep{kutoyants1997nonparametric,vanzanten2003empirical}.
\citet{vandervaart2005donsker} developed the corresponding Donsker theory,
with applications to the empirical distribution function and the local-time
density estimator, while \citet{kutoyants2010goodness} used these objects
to construct goodness-of-fit tests.

For finite observation horizons, \citet{aeckerle2021concentration}
established uniform concentration inequalities for diffusion empirical
processes and applied them to invariant-density estimation. They
identified two natural increment geometries arising from a
Poincar\'e-based Bernstein inequality: the uniform-norm geometry and the asymptotic-variance geometry. 

Building on these results, we provides a simultaneous observable-dependent bound that retains the local
complexity of both geometries, rather than replacing them with a single
class-wide supremum.

We now define the two scales. Let $0<C_P<\infty$ be the explicit
admissible Poincar\'e constant constructed in Lemma
\ref{lem:tool_loukianov_poincare} such that
\begin{align}
\operatorname{Var}_\pi(h)
\le
\frac{C_P}{2}
\int_{\RR}a(x)|h'(x)|^2\pi(x)dx
\label{eq:diffusion_poincare}
\end{align}
for every locally absolutely continuous $h\in L^2(\pi)$ for which the
right-hand side is finite.

For a bounded Borel observable $f$, we define
\begin{align*}
R_f(x):=
\int_x^\infty\overline f(y)\pi(y)dy
=
-\int_{-\infty}^x\overline f(y)\pi(y)dy,\ \cV_\pi(f)^2
:=
4\int_{\RR}\frac{R_f(x)^2}{\ell(x)}dx.
\end{align*}
The covariance formula of \citet{vandervaart2005donsker} identifies
$\cV_\pi(f)^2$ as the scalar central-limit variance. Appendix
\ref{app:proofs_diffusion} checks the normalization and finiteness, proves
the seminorm property, and combines that formula with a finite-time
Bernstein bound to identify the same quantity as the long-run variance.

Thus, for $f,h\in\cF$, we define the following two pseodu-metrics:
\begin{align*}
d_1(f,h)
:=
\frac{C_P}{\sqrt\tau}
\left\|(f-h)-\pi(f-h)\right\|_\infty,\ 
d_2(f,h):=
\sqrt2\,\cV_\pi(f-h),
\end{align*}
where $\|g\|_\infty:=\sup_{x\in\RR}|g(x)|$.

\begin{assumption}
\label{ass:diffusion_index_class}
The set $\cF$ is a standard Borel space whose elements are bounded Borel
functions on $\RR$. It contains the zero function, and the evaluation map
$(f,x)\mapsto f(x)$ is jointly Borel on $\cF\times\RR$. The maps $d_1$
and $d_2$ are jointly Borel, and $(\cF,d_1)$ is separable.
\end{assumption}

\begin{proposition}[Mixed-tail property of the occupation process]
\label{prop:diffusion_mixed_tail}
Under Assumptions \ref{ass:diffusion_coefficients} and
\ref{ass:diffusion_index_class}, $d_1$ and $d_2$ are finite-valued
pseudo-metrics and $(\cF,d_1+d_2)$ is separable. The map
$(\omega,f)\mapsto Z_\tau(f)(\omega)$ is jointly measurable, its sample
paths are almost surely $(d_1+d_2)$-continuous, and $Z_\tau(0)=0$ almost
surely. Moreover, for every $f,h\in\cF$ and $u\ge0$,
\begin{align}
\PP_\pi\cbr{
|Z_\tau(f)-Z_\tau(h)|
>
u\,d_1(f,h)+\sqrt u\,d_2(f,h)
}
\le2e^{-u}.
\label{eq:diffusion_mixed_tail}
\end{align}
Consequently, $\cbr{Z_\tau(f)}_{f\in\cF}$ is an anchored two-metric
mixed-tail process of orders $(\alpha_1,\alpha_2)=(1,2)$.
\end{proposition}

Fix Borel probability measures $\mu_1$ and $\mu_2$ on $\cF$ before the
trajectory is observed. For $j\in\cbr{1,2}$, let
\begin{align*}
v_j(f):=
d_j(f,0),\ 
B_{d_j}(f,r):=
\cbr{h\in\cF:d_j(f,h)\le r}.
\end{align*}
Thus, we have $v_1(f)=
\frac{C_P}{\sqrt\tau}\|f-\pi(f)\|_\infty,\ 
v_2(f)=
\sqrt2\,\cV_\pi(f)$. With the convention $\log(1/0):=+\infty$, we define
\begin{align*}
\Phi_1(f):=
\int_0^{4v_1(f)}
\log\frac{1}{\mu_1\rbr{B_{d_1}(f,r)}}dr,\ 
\Phi_2(f):=
\int_0^{4v_2(f)}
\left(
\log\frac{1}{\mu_2\rbr{B_{d_2}(f,r)}}
\right)^{1/2}dr.
\end{align*}
Applying Theorem \ref{thm:mixed_tail_majorization} to the diffusion empirical process $\cbr{Z_{\tau}(f)}_{f\in\cF}$, we have the following corollary.
\begin{corollary}[Simultaneous pointwise occupation bound]
\label{thm:diffusion_pointwise_majorization}
Suppose that Assumptions \ref{ass:diffusion_coefficients} and
\ref{ass:diffusion_index_class} hold. Fix $\tau>0$, start the diffusion
from $X_0\sim\pi$, and fix $\mu_1,\mu_2$ before observing the trajectory.
There exists a universal constant $C_{1,2}<\infty$ such that, for every
$\delta\in(0,1)$, with probability at least $1-\delta$, simultaneously
for every $f\in\cF$,
\begin{align}
|Z_\tau(f)|
\le
C_{1,2}
\cbr{
\Phi_1(f)+\Phi_2(f)
+v_1(f)\log\frac{e}{\delta}
+v_2(f)\sqrt{\log\frac{e}{\delta}}
}.
\label{eq:diffusion_pointwise_bound}
\end{align}
If either local complexity term is infinite at $f$, the bound at that
index is understood to be vacuous.
\end{corollary}

The terms $\Phi_1(f)$ and $\Phi_2(f)$ measure the local complexity of the
observable class in the uniform-amplitude and long-run
variance geometries, respectively. The last two terms retain the fixed-observable
Bernstein scales $\frac{C_P}{\sqrt\tau}
\|f-\pi(f)\|_\infty\log\frac{e}{\delta}
+
\sqrt2\,\cV_\pi(f)
\sqrt{\log\frac{e}{\delta}}$. Thus, the theorem does not replace the two scales of an individual
observable with the worst-case diameters of $\cF$.

Finally, notice that $|\widehat\pi_\tau(f)-\pi(f)|
=
\frac{|Z_\tau(f)|}{\sqrt\tau}$. Hence, Inequality \eqref{eq:diffusion_pointwise_bound} also provides a simultaneous bound
for all ergodic averages. Because the event is common to every $f$, it may
be evaluated at a measurable trajectory-dependent choice
$\widehat f\in\cF$ without another union bound; only the reference measures
must have been fixed in advance. All proofs are given in Appendix
\ref{app:proofs_diffusion}.

\section{Pointwise Envelopes for Decoupled Gaussian Chaos Processes}
\label{sec:higher_order_gaussian_chaos}

In this section, we provide our second application of the multi-regime
pointwise majorization theorem. We apply Theorem
\ref{thm:multi_regime_majorization} to a tensor-indexed family of
homogeneous Gaussian polynomials and derive a simultaneous pointwise
envelope for the associated decoupled Gaussian chaos process.

decoupled Gaussian chaoses arise naturally in the study of Gaussian
polynomials, multiple stochastic integrals, canonical $U$-statistics, and
random tensors
\citep{delapena1999decoupling,major2014multiple,bamberger2022hanson}.
Unlike a Gaussian process, a chaos of order larger than one is generally
not Gaussian. Its deviations are governed by several deterministic tensor
norms, rather than by a single variance metric. The moment estimates of
\citet{latala2006estimates} identify these norms through the partitions of
the coordinate set. This finite family of tail regimes makes higher-order
Gaussian chaos a natural application of Theorem
\ref{thm:multi_regime_majorization}.

Fix an integer $q\ge2$ and positive integers
$n_1,\ldots,n_q$. For every $r\in\cbr{1,\ldots,q}$, let
\begin{align*}
g^{(r)}
=
\rbr{g^{(r)}_1,\ldots,g^{(r)}_{n_r}}^\top
\sim N(0,I_{n_r}),
\end{align*}
and assume that $g^{(1)},\ldots,g^{(q)}$ are independent. We write
\begin{align*}
 [n]
 :=
 \cbr{1,\ldots,n}
 \quad\text{for every positive integer }n,
 \qquad
\mathcal I
:=
[n_1]\times\cdots\times[n_q].
\end{align*}
For a multi-index
$\boldsymbol i=(i_1,\ldots,i_q)\in\mathcal I$ and a nonempty subset
$I\subseteq[q]$, we write $\boldsymbol i_I
:=
(i_r)_{r\in I}$.

We let $(T,\mathcal A)$ be a standard Borel space, and let $t
\mapsto A_t=\rbr{a_{\boldsymbol i}(t)}_{\boldsymbol i\in\mathcal I}$ be a Borel measurable map from $T$ to the finite-dimensional tensor
space $\RR^{n_1\times\cdots\times n_q}$. Following the setting in Section \ref{sec:multi_regime_majorization}, we assume that there exists an
anchor $t_0\in T$ such that $A_{t_0}=0$.

\begin{definition}
\label{def:higher_order_gaussian_chaos}
The \emph{decoupled Gaussian chaos process of order $q$} indexed by
$\cbr{A_t}_{t\in T}$ is defined by
\begin{align}
C_t^{(q)}
:=
\sum_{\boldsymbol i\in\mathcal I}
a_{\boldsymbol i}(t)
\prod_{r=1}^qg^{(r)}_{i_r},
\ t\in T.
\label{eq:higher_order_gaussian_chaos_definition}
\end{align}
\end{definition}

Since the Gaussian vectors are independent and centered, we have
$\EE[C_t^{(q)}]=0,
\ t\in T$. Moreover, $A_{t_0}=0$ implies that
$C_{t_0}^{(q)}=0$ almost surely.

Now, we define the pseudo-metrics that we use for higher order Gaussian chaos. The metric that we use is the so-called partition norm introduced by \citet{latala2006estimates}. Let $\mathfrak P_q$ denote the collection of all partitions of $[q]$.
Its cardinality is the $q$th Bell number \citep{bell1938iterated}, denoted by $\mathsf B_q$. Therefore, any partition $\cP$ can be represented by $\mathcal P=\cbr{I_1,\ldots,I_k}\in\mathfrak P_q$ such that
$\cup_{i=1}^kI_i=[q],\ I_j\cap I_l=\varnothing$.

For every nonempty subset
$I\subseteq[q]$, let $\mathcal I_I
:=
\prod_{\ell\in I}[n_\ell]$. Thus, if
$\boldsymbol i=(i_1,\ldots,i_q)\in\mathcal I$, then
$\boldsymbol i_I=(i_\ell)_{\ell\in I}$ belongs to $\mathcal I_I$.

For every block $I_r$, the corresponding test array is $x^{(r)}
=
\rbr{x^{(r)}_{\boldsymbol j}}_{\boldsymbol j\in\mathcal I_{I_r}}
\in
\RR^{\mathcal I_{I_r}}$.

In particular, $x^{(r)}$ has
$\prod_{\ell\in I_r}n_\ell$ coordinates, indexed by the tuples
$\boldsymbol j=(j_\ell)_{\ell\in I_r}$. Since the blocks
$I_1,\ldots,I_k$ are pairwise disjoint and have union $[q]$, every full
multi-index $\boldsymbol i\in\mathcal I$ determines exactly one
coordinate $x^{(r)}_{\boldsymbol i_{I_r}}$ in each test array.

For a tensor
$A=(a_{\boldsymbol i})_{\boldsymbol i\in\mathcal I}
\in\RR^{n_1\times\cdots\times n_q}$, for any partition $\cP=\cbr{I_1,\cdots,I_k}$, we define
\begin{align}
\|A\|_{\mathcal P}
:=
\sup\cbr{
\sum_{\boldsymbol i\in\mathcal I}
a_{\boldsymbol i}
\prod_{r=1}^k
x^{(r)}_{\boldsymbol i_{I_r}}
:
\sum_{\boldsymbol j\in\mathcal I_{I_r}}
\rbr{x^{(r)}_{\boldsymbol j}}^2
\le1,
\ 
r\in\cbr{1,\ldots,k}
}.
\label{eq:gaussian_chaos_partition_norm}
\end{align}
The value in \eqref{eq:gaussian_chaos_partition_norm} does not depend
on the ordering of the blocks of $\mathcal P$. If the partition has one
block, then $
\|A\|_{\cbr{[q]}}
=
\left(
\sum_{\boldsymbol i\in\mathcal I}a_{\boldsymbol i}^2
\right)^{1/2}$, which is the Frobenius norm of the tensor. At the opposite extreme, the
partition into $q$ singleton blocks gives the injective norm of the
associated $q$-linear form.

Let $L_q$ be the constant in Lemma
\ref{lem:tool_latala_gaussian_chaos_moments}, and set
$K_q
=
e\,2^qL_q$. For every $\mathcal P\in\mathfrak P_q$, define
\begin{align}
d_{\mathcal P}(s,t)
:=
K_q\|A_s-A_t\|_{\mathcal P},
\ 
\alpha_{\mathcal P}
:=
\frac{2}{|\mathcal P|},
\label{eq:gaussian_chaos_regime_metrics}
\end{align}
and put $\rho(s,t)
:=
\sum_{\mathcal P\in\mathfrak P_q}d_{\mathcal P}(s,t)$.

The moment estimate of \citet{latala2006estimates} gives the following
multi-regime increment property.

\begin{proposition}
\label{prop:higher_order_gaussian_chaos_multi_regime}
The maps
$\cbr{d_{\mathcal P}}_{\mathcal P\in\mathfrak P_q}$ are
finite-valued, jointly Borel pseudo-metrics, and $(T,\rho)$ is
separable. For every $s,t\in T$ and every $u\ge0$,
\begin{align}
\PP\cbr{
\left|C_s^{(q)}-C_t^{(q)}\right|
>
\sum_{\mathcal P\in\mathfrak P_q}
u^{|\mathcal P|/2}d_{\mathcal P}(s,t)
}
\le
2e^{-u}.
\label{eq:higher_order_gaussian_chaos_increment}
\end{align}
The process $(C_t^{(q)})_{t\in T}$ is jointly measurable and has
almost surely $\rho$-continuous sample paths. Consequently, it is an
anchored $\mathsf B_q$-regime mixed-tail process with orders
$\alpha_{\mathcal P}
=
\frac{2}{|\mathcal P|},
\ 
\mathcal P\in\mathfrak P_q$.
\end{proposition}

For every
$\mathcal P\in\mathfrak P_q$, let $\mu_{\mathcal P}$ be a probability measure
on $T$. We define the pointwise anchor radii
\begin{align}
v_{\mathcal P}(t)
:=
d_{\mathcal P}(t,t_0)
=
K_q\|A_t\|_{\mathcal P}.
\label{eq:gaussian_chaos_anchor_radii}
\end{align}
For $r\ge0$, write $B_{d_{\mathcal P}}(t,r)
:=
\cbr{s\in T:d_{\mathcal P}(s,t)\le r}$.

The pointwise Fernique-Talagrand functional associated with
$\mathcal P$ is
\begin{align}
\Phi_{\mathcal P}(t)
&:=
\int_0^{4v_{\mathcal P}(t)}
\left(
\log\frac{1}{
\mu_{\mathcal P}(B_{d_{\mathcal P}}(t,r))
}
\right)^{\frac{|\mathcal P|}{2}}dr
=
K_q
\int_0^{4\|A_t\|_{\mathcal P}}
\left(
\log\frac{1}{
\mu_{\mathcal P}
\cbr{
s\in T:
\|A_s-A_t\|_{\mathcal P}\le r
}
}
\right)^{|\mathcal P|/2}dr.
\label{eq:gaussian_chaos_pointwise_functionals}
\end{align}

Applying Theorem \ref{thm:multi_regime_majorization} gives the
following pointwise envelope for the higher order Gaussian chaos.

\begin{theorem}
\label{thm:higher_order_gaussian_chaos_pointwise_envelope}
There exists a finite constant $C_q$, depending only on the chaos order
$q$, such that, for every $\delta\in(0,1)$, with probability at least
$1-\delta$, simultaneously for every $t\in T$,
\begin{align}
\left|C_t^{(q)}\right|
\le
C_q
\sum_{\mathcal P\in\mathfrak P_q}
\left[
\Phi_{\mathcal P}(t)
+
v_{\mathcal P}(t)
\left(
\log\frac{e}{\delta}
\right)^{|\mathcal P|/2}
\right].
\label{eq:higher_order_gaussian_chaos_pointwise_envelope}
\end{align}
\end{theorem}
For fixed $q$, decoupling inequalities show that the moments and tail probabilities of a diagonal-free coupled Gaussian chaos are controlled by those of its decoupled counterpart, up to constants depending only on $q$ \citep{delapena1994contraction}. Thus,
Theorem \ref{thm:higher_order_gaussian_chaos_pointwise_envelope} yields
an analogous upper envelope for diagonal-free coupled Gaussian chaos.

The proof of Proposition
\ref{prop:higher_order_gaussian_chaos_multi_regime} and Theorem
\ref{thm:higher_order_gaussian_chaos_pointwise_envelope} is deferred to
Appendix \ref{app:proofs_higher_order_gaussian_chaos}.

The functionals $\Phi_{\mathcal P}(t)$ in
\eqref{eq:higher_order_gaussian_chaos_pointwise_envelope} describe the
local multiscale complexity of the tensor family in the different
partition norms. The terms involving $v_{\mathcal P}(t)$ preserve the
deviation scales of the fixed-tensor moment inequality. More precisely,
a partition having $k$ blocks contributes to the confidence scale $\|A_t\|_{\mathcal P}
\left(
\log\frac{e}{\delta}
\right)^{k/2}$. The one-block partition gives the sub-Gaussian regime, whereas the
partition into $q$ singleton blocks gives the order-$2/q$ sub-Weibull
regime. The remaining partitions retain the intermediate tensor
contractions that occur in Lata{\l}a's moment inequality.

We conclude this section by making the second-order case explicit. Suppose
that $q=2$, and write
\begin{align*}
\mathcal P_{\mathrm F}
:=
\cbr{\cbr{1,2}},
\mathcal P_{\mathrm{op}}:=
\cbr{\cbr{1},\cbr{2}}.
\end{align*}
These are the only two elements of $\mathfrak P_2$. Moreover, when $q=2$, the tensor is just matrix and
\begin{align*}
\|A\|_{\mathcal P_{\mathrm F}}&=
\|A\|_{\mathrm F},\ 
\alpha_{\mathcal P_{\mathrm F}}=2,\\
\|A\|_{\mathcal P_{\mathrm{op}}}&=
\|A\|_{\mathrm{op}},\ 
\alpha_{\mathcal P_{\mathrm{op}}}=1.
\end{align*}
Thus, the process in Definition
\ref{def:higher_order_gaussian_chaos} becomes
\begin{align}
C_t^{(2)}
&=
\sum_{i=1}^{n_1}\sum_{j=1}^{n_2}
a_{ij}(t)g_i^{(1)}g_j^{(2)}
=
\rbr{g^{(1)}}^\top A_tg^{(2)},
\ t\in T.
\label{eq:decoupled_quadratic_gaussian_chaos}
\end{align}
This is the decoupled quadratic Gaussian chaos process associated with the
matrix family $\cbr{A_t}_{t\in T}$.

The corresponding pseudo-metrics and anchor radii are
\begin{align*}
d_{\mathcal P_{\mathrm F}}(s,t)
&=
K_2\|A_s-A_t\|_{\mathrm F},\ 
v_{\mathcal P_{\mathrm F}}(t)=
K_2\|A_t\|_{\mathrm F},\\
d_{\mathcal P_{\mathrm{op}}}(s,t)
&=
K_2\|A_s-A_t\|_{\mathrm{op}},\ 
v_{\mathcal P_{\mathrm{op}}}(t)=
K_2\|A_t\|_{\mathrm{op}}.
\end{align*}
For the two probability measures
$\mu_{\mathcal P_{\mathrm F}}$ and
$\mu_{\mathcal P_{\mathrm{op}}}$, the two pointwise
Fernique--Talagrand functionals in
\eqref{eq:gaussian_chaos_pointwise_functionals} reduce to
\begin{align}
\Phi_{\mathcal P_{\mathrm F}}(t)
&=
K_2
\int_0^{4\|A_t\|_{\mathrm F}}
\left(
\log\frac{1}{
\mu_{\mathcal P_{\mathrm F}}
\cbr{
s\in T:
\|A_s-A_t\|_{\mathrm F}\le r
}}
\right)^{1/2}dr,
\label{eq:quadratic_gaussian_chaos_frobenius_functional}\\
\Phi_{\mathcal P_{\mathrm{op}}}(t)
&=
K_2
\int_0^{4\|A_t\|_{\mathrm{op}}}
\log\frac{1}{
\mu_{\mathcal P_{\mathrm{op}}}
\cbr{
s\in T:
\|A_s-A_t\|_{\mathrm{op}}\le r
}}
\,dr.
\label{eq:quadratic_gaussian_chaos_operator_functional}
\end{align}
We have the following corollary.
\begin{corollary}[Decoupled quadratic Gaussian chaos]
\label{cor:quadratic_gaussian_chaos_pointwise_envelope}
There exists a universal constant $C_2<\infty$ such that, for every
$\delta\in(0,1)$, with probability at least $1-\delta$,
simultaneously for every $t\in T$,
\begin{align}
\left|C_t^{(2)}\right|
\le
C_2
\cbr{
\Phi_{\mathcal P_{\mathrm F}}(t)
+
\Phi_{\mathcal P_{\mathrm{op}}}(t)
+
v_{\mathcal P_{\mathrm F}}(t)
\sqrt{\log\frac{e}{\delta}}
+
v_{\mathcal P_{\mathrm{op}}}(t)
\log\frac{e}{\delta}
}.
\label{eq:quadratic_gaussian_chaos_pointwise_envelope}
\end{align}
\end{corollary}

The first two terms in
\eqref{eq:quadratic_gaussian_chaos_pointwise_envelope} describe the local
multiscale complexity around $A_t$ in the Frobenius and operator norms.
The last two terms retain the two confidence scales of a fixed Gaussian
bilinear form. They are essentially $\|A_t\|_{\mathrm F}
\sqrt{\log\frac{e}{\delta}}
+
\|A_t\|_{\mathrm{op}}
\log\frac{e}{\delta}$ up to the universal factor $K_2$. This specialized inequality can also be derived by combining Hanson-Wright inequality and Theorem \ref{thm:mixed_tail_majorization}.

\section{Discussion}\label{sec:discussion}
In this paper, we develop the first simultaneous pointwise majorization
theory for Banach-valued stochastic processes with 
finite-metric mixed-tail increments. The single-regime result extends
the finite Gaussian upper-envelope construction of \citet{xu2026} to
arbitrary positive sub-Weibull orders and separable index spaces; in the
Gaussian case, it also removes the auxiliary baseline scales and peeling
logarithms from the previous upper bound. Theorem \ref{thm:multi_regime_majorization}
retains, for every increment regime, a separate pseudo-metric, reference
measure, pointwise ball-mass functional, anchor radius, and tail order.
The proof constructs a measure-generated admissible chain for each
pseudo-metric, keeps its cost at every individual index, and synchronizes
the resulting chains through a nested common refinement. We apply the
theory to the occupation empirical process of a stationary scalar
diffusion, where a Bernstein inequality produces uniform-norm and
long-run-variance scales, and to higher-order decoupled Gaussian chaos,
where set partitions generate finitely many tensor-norm regimes. In both
applications, one event controls the full index class while the
deterministic bound adapts to the observable or tensor being evaluated.

Several directions merit further study. First, the simultaneous nature
of the bounds makes them potentially useful for analyzing empirical risk
minimization and other data-dependent selection procedures. The common
event may be evaluated at a measurable predictor selected from the same
sample without an additional union bound, provided that the index class,
anchor, increment pseudo-metrics, and reference measures are fixed in
advance or are handled through conditioning or sample splitting.
Converting this observation into excess-risk guarantees will require
combining pointwise majorization with appropriate curvature, margin, or
localization properties of the risk. It would be particularly
interesting to determine whether the resulting candidate-dependent
certificates can complement recent model-free risk-evaluation methods
based on refitting and perturbation
\citep{wainwright2025wild,hu2025perturbing,
hu2026interleaved,ni2026upper}.

A second direction concerns the interpretation and selection of the
ambient reference measures. A pointwise ball-mass functional is not a
property of the index class alone: it also depends on the chosen
pseudo-metric, anchor, reference measure, and index. Developing
principled ways to select or optimize the regime-specific measures while
preserving simultaneous validity may lead to useful statistical
complexity notions. Comparisons with local Rademacher complexities
\citep{bartlett2005local,koltchinskii2006local}, VC dimension
\citep{vapnik2013nature}, and spectral or eigendecay conditions
\citep{hu2025contextual} would clarify when pointwise localization yields
sharper learning guarantees.

Finally, it remains to understand the optimality of our
bound. The constant produced by the common-refinement argument depends
explicitly on the number of regimes and their tail orders, and it is
natural to ask for the sharp dependence on these parameters. More
generally, matching lower envelopes cannot follow from an upper
sub-Weibull or mixed-tail increment condition alone, because the same
condition is satisfied by degenerate processes. The Gaussian lower
bounds of \citet{xu2026} provide a natural starting point, but a general
lower-bound theory will require suitable nondegeneracy, small-ball, or
two-sided increment assumptions. Under such conditions, one may ask
whether the regime-specific ball-mass terms, the anchor-dependent
confidence scales, and their simultaneous separation are unavoidable.

\appendix
\section{Mathematical Tools}
\label{app:mathematical_tools}

We record only the external results used directly in the two applications.
The diffusion results are stated directly in the notation of
Section \ref{sec:pointwise_ergodic_diffusion}.

\begin{lemma}[\citet{lamperski2023nonasymptotic}]
\label{lem:tool_lamperski_hanson_wright}
Let $\zeta=(\zeta_1,\ldots,\zeta_n)^\top$ have independent, centered,
real-valued coordinates satisfying
$\max_i\|\zeta_i\|_{\psi_2}\le b$, where
\begin{align*}
\|Y\|_{\psi_2}
:=
\inf\cbr{
r>0:
\EE\exp\rbr{\frac{Y^2}{r^2}}\le2
}.
\end{align*}
For every nonzero deterministic real matrix $A$ and every
$\varepsilon\ge0$,
\begin{align*}
\PP\cbr{
\zeta^\top A\zeta-
\EE\sbr{\zeta^\top A\zeta}>\varepsilon
}
\le
2\exp\cbr{
-\frac1{c_{HW}}
\min\cbr{
\frac{\varepsilon^2}{b^4\|A\|_{\mathrm F}^2},
\frac{\varepsilon}{b^2\|A\|_{\mathrm{op}}}
}
},
\end{align*}
where $c_{HW}=2048$ is a universal constant.
\end{lemma}

\begin{lemma}[\citet{loukianov2011spectral}]
\label{lem:tool_loukianov_poincare}
Consider the scalar diffusion in \eqref{eq:diffusion_sde} under
Assumption \ref{ass:diffusion_coefficients}, with invariant density $\pi$
and $\ell=a\pi$. Define
\begin{align*}
F(x):=
\int_{-\infty}^x\pi(y)dy,\ 
\overline F(x):=
\int_x^\infty\pi(y)dy,
\end{align*}
and
$B_+:=
\sup_{x>0}
\overline F(x)
\int_0^x\frac{2}{\ell(y)}dy,\ 
B_-:=
\sup_{x<0}
F(x)
\int_x^0\frac{2}{\ell(y)}dy$.

If $B_+\vee B_-<\infty$, then the stationary diffusion is reversible
with respect to $\pi$, and we have the following Poincare inequality:
\begin{align}
\operatorname{Var}_\pi(h)
\le
\frac{C_P}{2}
\int_{\RR}a(x)|h'(x)|^2\pi(x)dx,
\ 
C_P:=4\max\cbr{B_+,B_-},
\label{eq:tool_loukianov_poincare}
\end{align}
for every locally absolutely continuous $h\in L^2(\pi)$ for which the
right-hand side is finite.
\end{lemma}

\begin{lemma}[\citet{gao2014bernstein}]
\label{lem:tool_gao_bernstein}
Let $X$ be the stationary diffusion in
Section \ref{sec:pointwise_ergodic_diffusion}. Suppose that it is reversible
with respect to $\pi$ and that,  $\exists\ 0<C_P<\infty$ such that
\begin{align}
\operatorname{Var}_\pi(h)
\le
\frac{C_P}{2}
\int_{\RR}a(x)|h'(x)|^2\pi(x)dx
\label{eq:tool_gao_poincare}
\end{align}
for every locally absolutely continuous $h\in L^2(\pi)$ for which the
right-hand side is finite. Let $g$ be bounded and Borel with $\pi(g)=0$.
Then, for every $t,u>0$,
\begin{align}
\PP_\pi\cbr{
\left|
\frac1{\sqrt t}\int_0^t g(X_s)ds
\right|
>
2\sqrt{C_P}\|g\|_{L^2(\pi)}\sqrt u
+
\frac{C_P\|g\|_\infty}{\sqrt t}u
}
\le2e^{-u}.
\label{eq:tool_gao_coarse_two_sided}
\end{align}
If the long-run variance of $g$ equals $\cV_\pi(g)^2$, i.e., 
$\cV_\pi(g)^2
=
\lim_{t\to\infty}
\frac1t
\operatorname{Var}_{\PP_\pi}
\left(
\int_0^t g(X_s)ds
\right)$, then we have that
\begin{align}
\cV_\pi(g)^2
\le
2C_P\|g\|_{L^2(\pi)}^2,
\label{eq:tool_gao_variance_domination}
\end{align}
and
\begin{align}
\PP_\pi\cbr{
\left|
\frac1{\sqrt t}\int_0^t g(X_s)ds
\right|
>
\cV_\pi(g)\sqrt{2u}
+
\frac{C_P\|g\|_\infty}{\sqrt t}u
}
\le2e^{-u}.
\label{eq:tool_gao_two_sided}
\end{align}
The two-sided statements follow by applying the one-sided result of
\citet{gao2014bernstein} to $g$ and $-g$.
\end{lemma}
\begin{lemma}[\citet{vandervaart2005donsker}]
\label{lem:tool_vdvz_scalar_clt}
Let $X$ be the stationary scalar diffusion considered in
Section \ref{sec:pointwise_ergodic_diffusion}. Let $f$ be a bounded
Borel function such that
\begin{align*}
\cV_\pi(f)^2
=
4\int_{\RR}\frac{R_f(x)^2}{\ell(x)}\,dx
<\infty.
\end{align*}
Then we have that
\begin{align*}
\frac{1}{\sqrt t}
\int_0^t\overline f(X_s)\,ds
\Longrightarrow
\mathcal N\rbr{0,\cV_\pi(f)^2},
\ t\to\infty.
\end{align*}
\end{lemma}
\begin{lemma}[\citet{latala2006estimates}]
\label{lem:tool_latala_gaussian_chaos_moments}
Fix an integer $q\ge2$ and positive integers
$n_1,\ldots,n_q$. Let $g^{(r)}
\sim
\mathcal N(0,I_{n_r}),
\qquad r\in\cbr{1,\ldots,q}$ be independent standard Gaussian vectors. We write
\begin{align*}
[n]
:=
\cbr{1,\ldots,n},
\ 
\mathcal I
:=
[n_1]\times\cdots\times[n_q],
\end{align*}
and let $\mathfrak P_q$ be the collection of all partitions of $[q]$.
For a deterministic tensor
$A=(a_{\boldsymbol i})_{\boldsymbol i\in\mathcal I}$ and a partition
$\mathcal P=\cbr{I_1,\ldots,I_k}\in\mathfrak P_q$, define
\begin{align*}
\|A\|_{\mathcal P}
:=
\sup\cbr{
\sum_{\boldsymbol i\in\mathcal I}
a_{\boldsymbol i}
\prod_{r=1}^k
x^{(r)}_{\boldsymbol i_{I_r}}:
\sum_{\boldsymbol i_{I_r}}
\rbr{x^{(r)}_{\boldsymbol i_{I_r}}}^2
\le1,
\quad r\in\cbr{1,\ldots,k}
}.
\end{align*}
Then there exists a finite constant $L_q$, depending only on $q$, such
that, for every $p\ge2$,
\begin{align}
\left\|
\sum_{\boldsymbol i\in\mathcal I}
a_{\boldsymbol i}
\prod_{r=1}^qg^{(r)}_{i_r}
\right\|_{L^p}
\le
L_q
\sum_{\mathcal P\in\mathfrak P_q}
p^{|\mathcal P|/2}\|A\|_{\mathcal P}.
\label{eq:tool_latala_gaussian_chaos_moments}
\end{align}
\end{lemma}

Lemma \ref{lem:tool_latala_gaussian_chaos_moments} is the upper half of
the two-sided moment estimate in Theorem~1 of
\citet{latala2006estimates}. Only this upper estimate is used below.

\section{Proofs in Section \ref{sec:mixed_tail-majorization}}\label{app:proofs_mixed_tail}

Throughout this appendix, we use the convention $\log\frac{1}{0}:=+\infty$. Changing the process on a null event does not affect any of the claims
below. We therefore redefine
\begin{align*}
Z_t(\omega):=0,\ \omega \notin\Omega_0,t\in T,
\end{align*}
Thus, every sample path used below is
$\rho$-continuous and satisfies $Z_{t_0}=0$.

We first verify the measurability of the pointwise radii and profiles.

\begin{lemma}
\label{lem:mixed_tail_measurability}
For every $j\in\cbr{1,2}$, the map
\begin{align*}
(t,r)\longmapsto\mu_j(B_{d_j}(t,r))
\end{align*}
is Borel measurable on $T\times[0,\infty)$. Consequently,
$v_j$ and $\Phi_j$ are $[0,\infty]$-valued Borel functions on $T$.
\end{lemma}

\begin{proof}[Proof of Lemma \ref{lem:mixed_tail_measurability}]
Fix $j\in\cbr{1,2}$. Since $d_j$ is jointly
$\mathcal A\otimes\mathcal A$-measurable, the map
\begin{align*}
(t,r,s)
\longmapsto
\mathbf 1_{\cbr{d_j(t,s)\le r}}
\end{align*}
is measurable from $\rbr{
T\times[0,\infty)\times T,
\mathcal A\otimes\mathcal B([0,\infty))\otimes\mathcal A
}$ to $\rbr{\RR,\mathcal B(\RR)}$. Therefore, by measurability of
parameterized integrals, the map $(t,r)
\mapsto
\int_T
\mathbf 1_{\cbr{d_j(t,s)\le r}}\mu_j(ds)=
\mu_j(B_{d_j}(t,r))$ is Borel measurable.
Moreover, since $t_0$ is fixed and $d_j$ is jointly Borel, the map $v_j(t)=d_j(t,t_0)$ is Borel measurable. Again, we define
\begin{align*}
Q_j(t,r)
:=
\begin{cases}
\displaystyle
\left(
\log\frac{1}{\mu_j(B_{d_j}(t,r))}
\right)^{1/\alpha_j},
&0\le r\le4v_j(t),\\[1.2ex]
0,
&r>4v_j(t).
\end{cases}
\end{align*}
The preceding measurability statements show that $Q_j$ is a
nonnegative, extended-real-valued Borel function. Notice that the
piecewise definition avoids the ambiguous product $0\cdot\infty$ when a
ball has zero $\mu_j$-mass. By the definition of $\Phi_j$, we have that $\Phi_j(t)
=
\int_0^\infty Q_j(t,r)dr$.
Another application of measurability of parameterized integrals proves
that $\Phi_j$ is Borel measurable. Since the same argument applies to
$j=1$ and $j=2$, we finish the proof.
\end{proof}

We next transfer each fixed ambient measure to an arbitrary finite
subset. This is the nearest-point construction that permits the
measure-generated chains to be built inside the finite set without
losing the ambient ball masses.

\begin{lemma}
\label{lem:mixed_tail_ambient_restriction}
Let $K\subseteq T$ be finite and contain $t_0$. For
$j\in\cbr{1,2}$, write $d_{j,K}:=d_j|_{K\times K}$.
There exists a probability measure $\widetilde\mu_{j,K}$ on $K$ such
that, for every $t\in K$ and every $r\ge0$,
\begin{align*}
\widetilde\mu_{j,K}(\cbr{t_0})
\ge\frac12,\ 
\widetilde\mu_{j,K}(B_{d_{j,K}}(t,2r))
\ge
\frac12\mu_j(B_{d_j}(t,r)).
\end{align*}
Moreover, if $c_{\alpha_j}^{(0)}
:=
\max\cbr{1,2^{1/\alpha_j-1}}$, then
\begin{align*}
&\int_0^{4v_j(t)}
\left(
\log\frac{1}{
\widetilde\mu_{j,K}(B_{d_{j,K}}(t,r))
}
\right)^{1/\alpha_j}dr\le
2c_{\alpha_j}^{(0)}
\Phi_j(t)
+
4c_{\alpha_j}^{(0)}
(\log2)^{1/\alpha_j}v_j(t).
\end{align*}
\end{lemma}

\begin{proof}[Proof of Lemma
\ref{lem:mixed_tail_ambient_restriction}]
Fix $j\in\cbr{1,2}$ and enumerate
\begin{align*}
K=\cbr{k_1,\ldots,k_N}.
\end{align*}
For every $s\in T$, let $q_{j,K}(s)$ be a point in $K$ minimizing
$d_j(s,\cdot)$, and break ties by choosing the point with the smallest
index. More explicitly, define
\begin{align*}
C_i
:={}&
\bigcap_{m<i}
\cbr{s\in T:d_j(s,k_i)<d_j(s,k_m)}
\cap
\bigcap_{m>i}
\cbr{s\in T:d_j(s,k_i)\le d_j(s,k_m)}.
\end{align*}
Every set $C_i$ is Borel measurable because $d_j$ is jointly Borel.
The sets $C_1,\ldots,C_N$ form a partition of $T$, and
$q_{j,K}(s)=k_i$ for $s\in C_i$. Hence, $q_{j,K}:T\to K$ is Borel
measurable and satisfies
\begin{align*}
d_j(s,q_{j,K}(s))
=
d_j(s,K)
:=
\min_{k\in K}d_j(s,k).
\end{align*}

Define
\begin{align*}
\widetilde\mu_{j,K}
:=
\frac12\delta_{t_0}
+
\frac12(q_{j,K})_\sharp\mu_j,
\end{align*}
where the pushforward measure is defined by
\begin{align*}
(q_{j,K})_\sharp\mu_j(A)
:=
\mu_j(q_{j,K}^{-1}(A)),
\ A\subseteq K.
\end{align*}
By construction, we have
\begin{align*}
\widetilde\mu_{j,K}(\cbr{t_0})\ge\frac12.
\end{align*}

We now prove the ball-mass comparison. Fix $t\in K$ and $r\ge0$.
If $s\in B_{d_j}(t,r)$, then $t\in K$ and the nearest-point property
give
\begin{align*}
d_j(s,q_{j,K}(s))
\le
d_j(s,t)
\le r.
\end{align*}
By the triangle inequality, we have
\begin{align*}
d_j(t,q_{j,K}(s))
&\le
d_j(t,s)+d_j(s,q_{j,K}(s))\le2r.
\end{align*}
Therefore, we get $q_{j,K}(B_{d_j}(t,r))
\subseteq
B_{d_{j,K}}(t,2r)$.
It follows that
\begin{align*}
\widetilde\mu_{j,K}(B_{d_{j,K}}(t,2r))\ge
\frac12
(q_{j,K})_\sharp\mu_j(B_{d_{j,K}}(t,2r))=
\frac12
\mu_j(q_{j,K}^{-1}(B_{d_{j,K}}(t,2r)))\ge
\frac12\mu_j(B_{d_j}(t,r)).
\end{align*}

It remains to prove the integral comparison. Recall that
$p_j:=\frac{1}{\alpha_j}$.

By the change of variable $r=2u$, we have
\begin{align*}
&\int_0^{4v_j(t)}
\left(
\log\frac{1}{
\widetilde\mu_{j,K}(B_{d_{j,K}}(t,r))
}
\right)^{p_j}dr=
2\int_0^{2v_j(t)}
\left(
\log\frac{1}{
\widetilde\mu_{j,K}(B_{d_{j,K}}(t,2u))
}
\right)^{p_j}du.
\end{align*}
The ball-mass comparison gives us
\begin{align*}
\log\frac{1}{
\widetilde\mu_{j,K}(B_{d_{j,K}}(t,2u))
}\le
\log\frac{2}{\mu_j(B_{d_j}(t,u))}=
\log\frac{1}{\mu_j(B_{d_j}(t,u))}
+
\log2.
\end{align*}
For nonnegative $x,y$, we have
\begin{align*}
(x+y)^{p_j}
\le
c_{\alpha_j}^{(0)}
\rbr{x^{p_j}+y^{p_j}}.
\end{align*}
Finally, combining these parts together, we obtain
\begin{align*}
\int_0^{4v_j(t)}
\left(
\log\frac{1}{
\widetilde\mu_{j,K}(B_{d_{j,K}}(t,r))
}
\right)^{p_j}dr&\le
2c_{\alpha_j}^{(0)}
\int_0^{2v_j(t)}
\left(
\log\frac{1}{\mu_j(B_{d_j}(t,u))}
\right)^{p_j}du+
2c_{\alpha_j}^{(0)}
\int_0^{2v_j(t)}(\log2)^{p_j}du\\
&\le
2c_{\alpha_j}^{(0)}\Phi_j(t)
+
4c_{\alpha_j}^{(0)}
(\log2)^{p_j}v_j(t),
\end{align*}
where we used $2v_j(t)\le4v_j(t)$ in the final inequality. This proves
the assertion for $j=1,2$. Applying the same construction to
both metrics finishes the proof.
\end{proof}

We now construct, for each metric, an admissible sequence whose
chaining cost remains pointwise rather than being replaced by a
worst-case supremum over $K$.

\begin{lemma}
\label{lem:mixed_tail_pointwise_measure_chain}
Let $K\subseteq T$ be finite and contain $t_0$. For every
$j\in\cbr{1,2}$, there exists an admissible sequence $\mathcal T_j=(T_{j,k})_{k\ge0}$ of subsets of $K$ such that $T_{j,0}=\cbr{t_0}$ and $T_{j,k}=K$ for all sufficiently large $k$. Moreover, for every
$t\in K$,
\begin{align*}
G_{j,K}(t):=
\sum_{k\ge0}
2^{k/\alpha_j}d_j(t,T_{j,k})\le
C_{\alpha_j}^{(1)}
\rbr{\Phi_j(t)+v_j(t)},
\end{align*}
where $C_{\alpha_j}^{(1)}<\infty$ depends only on $\alpha_j$.
\end{lemma}

\begin{proof}[Proof of Lemma
\ref{lem:mixed_tail_pointwise_measure_chain}]
Fix $j\in\cbr{1,2}$ and write $p_j:=\frac{1}{\alpha_j}$.
Let $\widetilde\mu_{j,K}$ be the probability measure supplied by
Lemma \ref{lem:mixed_tail_ambient_restriction}. For $t\in K$ and
$n\ge0$, define
\begin{align*}
r_{j,n}(t)
:=
\min\cbr{
r\in\cbr{d_j(t,s):s\in K}:
\widetilde\mu_{j,K}(B_{d_{j,K}}(t,r))
\ge2^{-2^n}
}.
\end{align*}
The minimum exists because $K$ is finite and the largest
$d_{j,K}$-ball centered at $t$ has
$\widetilde\mu_{j,K}$-mass one. Since the mass threshold
$2^{-2^n}$ is nonincreasing in $n$, the sequence
$\rbr{r_{j,n}(t)}_{n\ge0}$ is nonincreasing.

For a fixed $n\ge0$, we order the points of $K$ by nondecreasing
$r_{j,n}(t)$ and break ties according to an arbitrary fixed enumeration
of $K$. Traverse the ordered list and retain a point $t$ precisely when
\begin{align*}
d_j(t,s)
>
r_{j,n}(t)+r_{j,n}(s)
\end{align*}
for every previously retained point $s$. Denote the set of retained
points by $S_{j,n}$.

We first prove that the balls $\cbr{
B_{d_{j,K}}(s,r_{j,n}(s)):
s\in S_{j,n}
}$ are pairwise disjoint. Suppose that two such balls, centered at
distinct points $s,t\in S_{j,n}$, intersect. Then there exists
$z\in K$ such that
\begin{align*}
d_j(s,z)\le r_{j,n}(s),
\ 
d_j(t,z)\le r_{j,n}(t).
\end{align*}
Applying the triangle inequality would gives us
\begin{align*}
d_j(s,t)
&\le
d_j(s,z)+d_j(z,t)\le
r_{j,n}(s)+r_{j,n}(t).
\end{align*}
On the other hand, because both $s$ and $t$ were retained, the greedy
rule gives
\begin{align*}
d_j(s,t)
>
r_{j,n}(s)+r_{j,n}(t),
\end{align*}
which is a contradiction.

By the definition of $r_{j,n}(s)$, every disjoint ball
has $\widetilde\mu_{j,K}$-mass at least $2^{-2^n}$. Hence, we have
\begin{align*}
1\ge
\sum_{s\in S_{j,n}}
\widetilde\mu_{j,K}
\rbr{B_{d_{j,K}}(s,r_{j,n}(s))}\ge
|S_{j,n}|2^{-2^n}.
\end{align*}
Consequently, we get $|S_{j,n}|
\le
2^{2^n}$. We next prove that $S_{j,n}$ approximates every point at its own mass
radius. If $t\in S_{j,n}$, then
$d_j(t,S_{j,n})=0$. If $t\notin S_{j,n}$, the point $t$ was rejected
because there exists a previously retained point $s\in S_{j,n}$ such
that $r_{j,n}(s)\le r_{j,n}(t)$ and
\begin{align*}
d_j(t,s)
\le
r_{j,n}(t)+r_{j,n}(s).
\end{align*}
Therefore, in both cases, we have
\begin{align}
d_j(t,S_{j,n})
\le
2r_{j,n}(t),
\ t\in K.
\label{eq:mixed_mass_net_approximation}
\end{align}
For $t\in K$, we define
\begin{align*}
F_{j,t}(r)
:=
\left(
\log\frac{1}{
\widetilde\mu_{j,K}(B_{d_{j,K}}(t,r))
}
\right)^{p_j}.
\end{align*}
If $r<r_{j,n}(t)$, by the definition of
$r_{j,n}(t)$, we have that $\widetilde\mu_{j,K}(B_{d_{j,K}}(t,r))
<
2^{-2^n}$, which implies that
\begin{align}
F_{j,t}(r)
\ge
(\log2)^{p_j}2^{np_j}.
\label{eq:mixed_mass_integrand_lower_bound}
\end{align}
Since the intervals$[r_{j,n+1}(t),r_{j,n}(t)),
\ n\ge0$ are pairwise disjoint and contained in
$[0,r_{j,0}(t)]$, applying inequality
\eqref{eq:mixed_mass_integrand_lower_bound} gives us
\begin{align}
\int_0^{r_{j,0}(t)}F_{j,t}(r)dr
&\ge
(\log2)^{p_j}
\sum_{n\ge0}
2^{np_j}
\rbr{r_{j,n}(t)-r_{j,n+1}(t)}.
\label{eq:mixed_mass_integral_radius_differences}
\end{align}
We next compare the last sum with the weighted sum of the mass radii.
If the integral on the left-hand side is infinite, the desired
estimate is immediate. Suppose that it is finite. Then we have $\widetilde\mu_{j,K}(B_{d_{j,K}}(t,0))>0$. 
Indeed, if this mass were zero, the finiteness of $K$ would imply that
$F_{j,t}$ is infinite on an interval of positive length and the conclusion is trivial. 

Since $2^{-2^n}\downarrow0$ as $n$ grows to infinity, the positive mass at radius zero implies that
$r_{j,n}(t)=0$ for sufficiently large $n$.

We may therefore apply summation by parts and obtain
\begin{align*}
&\sum_{n\ge0}
2^{np_j}
\rbr{r_{j,n}(t)-r_{j,n+1}(t)}=
\rbr{1-2^{-p_j}}
\sum_{n\ge0}2^{np_j}r_{j,n}(t)
+
2^{-p_j}r_{j,0}(t)\ge
\rbr{1-2^{-p_j}}
\sum_{n\ge0}2^{np_j}r_{j,n}(t).
\end{align*}
Combining this identity with
\eqref{eq:mixed_mass_integral_radius_differences}, we obtain
\begin{align}
\sum_{n\ge0}2^{np_j}r_{j,n}(t)
\le
\frac{1}{
(\log2)^{p_j}(1-2^{-p_j})
}
\int_0^{r_{j,0}(t)}F_{j,t}(r)dr.
\label{eq:mixed_weighted_mass_radius_bound}
\end{align}

We now define the admissible sequence. Let $T_{j,0}:=\cbr{t_0},
\ 
T_{j,k}:=S_{j,k-1},
\ k\ge1$. Thus, we have that $|T_{j,0}|=1,
\ 
|T_{j,k}|
\le
2^{2^{k-1}}
\le
2^{2^k}$. Thus, the sequence is admissible. Since $K$ is finite, choose
$k_K\ge1$ such that $2^{2^{k_K}}\ge|K|$ and replace $T_{j,k}$ by $K$ for every $k\ge k_K$. This replacement
only decreases $d_j(t,T_{j,k})$, thus it preserves admissibility and ensures
that $T_{j,k}=K$ for sufficiently large $k$.

Since $T_{j,0}=\cbr{t_0}$, we have
\begin{align*}
d_j(t,T_{j,0})
=
d_j(t,t_0)
=
v_j(t).
\end{align*}
Moreover, by \eqref{eq:mixed_mass_net_approximation}, for every
$k\ge1$, we get
\begin{align*}
d_j(t,T_{j,k})
\le
2r_{j,k-1}(t).
\end{align*}
This inequality remains valid at the levels at which $T_{j,k}$ has
been replaced by $K$, because the left-hand side is then zero.
Consequently, we have
\begin{align*}
G_{j,K}(t)
&=
\sum_{k\ge0}
2^{kp_j}d_j(t,T_{j,k})=
v_j(t)
+
\sum_{k\ge1}
2^{kp_j}d_j(t,T_{j,k})\\
&\le
v_j(t)
+
2\sum_{k\ge1}
2^{kp_j}r_{j,k-1}(t)=
v_j(t)
+
2^{1+p_j}
\sum_{n\ge0}
2^{np_j}r_{j,n}(t).
\end{align*}
Using \eqref{eq:mixed_weighted_mass_radius_bound}, we conclude that
\begin{align}
G_{j,K}(t)
&\le
v_j(t)
+
\frac{2^{1+p_j}}{
(\log2)^{p_j}(1-2^{-p_j})
}
\int_0^{r_{j,0}(t)}F_{j,t}(r)dr.
\label{eq:mixed_chain_cost_restricted_integral}
\end{align}

The atom at the anchor now gives the required pointwise truncation.
The closed ball $B_{d_{j,K}}(t,v_j(t))$ contains $t_0$. Since
\begin{align*}
\widetilde\mu_{j,K}(\cbr{t_0})\ge\frac12=2^{-2^0},
\end{align*}
the definition of $r_{j,0}(t)$ implies that $r_{j,0}(t)\le v_j(t)$.
Therefore, we have
\begin{align*}
\int_0^{r_{j,0}(t)}F_{j,t}(r)dr\le
\int_0^{4v_j(t)}
\left(
\log\frac{1}{
\widetilde\mu_{j,K}(B_{d_{j,K}}(t,r))
}
\right)^{p_j}dr\le
2c_{\alpha_j}^{(0)}\Phi_j(t)
+
4c_{\alpha_j}^{(0)}
(\log2)^{p_j}v_j(t),
\end{align*}
where the final inequality follows from
Lemma \ref{lem:mixed_tail_ambient_restriction}. Inserting this estimate
into \eqref{eq:mixed_chain_cost_restricted_integral} and absorbing the
constants into $C_{\alpha_j}^{(1)}$ proves
\begin{align*}
G_{j,K}(t)
\le
C_{\alpha_j}^{(1)}
\rbr{\Phi_j(t)+v_j(t)}.
\end{align*}
The construction is deterministic, and the inequality holds
simultaneously for every $t\in K$. This holds for $j=1$ and $j=2$. We finish the proof.
\end{proof}

We next combine the two admissible sequences. The common refinement is essential because the increment assumption controls the sum of the two
metric contributions for a single process.

\begin{lemma}
\label{lem:mixed_tail_finite_pointwise}
Let $K\subseteq T$ be finite and contain $t_0$. For every
$j\in\cbr{1,2}$, let $\mathcal T_j=(T_{j,k})_{k\ge0}$ be an admissible sequence of subsets of $K$ such that $T_{j,0}=\cbr{t_0}$ and $T_{j,k}=K$ for all sufficiently large $k$. We define
\begin{align*}
G_j(t)
:=
\sum_{k\ge0}
2^{k/\alpha_j}d_j(t,T_{j,k}),
\ t\in K.
\end{align*}
There exists a finite constant $A_{\alpha_1,\alpha_2}$, depending only
on $\alpha_1$ and $\alpha_2$, such that, for every $u\ge1$, with
probability at least $1-e^{-u}$, simultaneously for every $t\in K$,
\begin{align*}
\|Z_t\|
\le
A_{\alpha_1,\alpha_2}
\cbr{
G_1(t)+G_2(t)
+
u^{1/\alpha_1}v_1(t)
+
u^{1/\alpha_2}v_2(t)
}.
\end{align*}
\end{lemma}

\begin{proof}[Proof of Lemma
\ref{lem:mixed_tail_finite_pointwise}]
For $j\in\cbr{1,2}$, $k\ge0$, and $t\in K$, define the nearest point map $p_{j,k}(t)\in T_{j,k}$ such that
\begin{align*}
d_j(t,p_{j,k}(t))
=
d_j(t,T_{j,k}).
\end{align*}
Such a point exists because $T_{j,k}$ is a nonempty finite set. When
$T_{j,k}=K$, we choose $p_{j,k}(t):=t$.

For every $n\ge2$, define an equivalence relation on $K$ by
\begin{align*}
s\sim_n t\ \text{iff}\ p_{j,k}(s)=p_{j,k}(t)
\end{align*}
for every $j\in\cbr{1,2}$ and every $0\le k\le n-2$. Let
$\mathcal A_n$ be the corresponding partition of $K$, and we write
$\mathcal A_n(t)$ for the cell containing $t$. Equivalently, we define
\begin{align*}
P_n(t)
:=
\rbr{
p_{1,0}(t),\ldots,p_{1,n-2}(t),
p_{2,0}(t),\ldots,p_{2,n-2}(t)
}.
\end{align*}
Then two points lie in the same cell of $\mathcal A_n$ precisely when
they have the same image under $P_n$.

The partitions are nested. Indeed, if $s\sim_n t$, then the nearest
point labels of $s$ and $t$ agree at all levels $0\le k\le n-2$, and
therefore they also agree at all levels $0\le k\le n-3$. Hence,
\begin{align*}
s\sim_n t
\Rightarrow
s\sim_{n-1}t.
\end{align*}
Thus, $\mathcal A_n$ refines $\mathcal A_{n-1}$ for every $n\ge3$.

We next bound the number of cells. By admissibility of the two
sequences, we have that
\begin{align*}
|\mathcal A_n|\le
\prod_{j=1}^2
\prod_{k=0}^{n-2}
|T_{j,k}|\le
\prod_{j=1}^2
\prod_{k=0}^{n-2}
2^{2^k}=
2^{2\sum_{k=0}^{n-2}2^k}=
2^{2(2^{n-1}-1)}=
2^{2^n-2}\le
2^{2^n}.
\end{align*}

For every $n\ge2$, define the pointwise combined approximation cost
\begin{align*}
R_n(t)
:=
2^{n/\alpha_1}d_1(t,T_{1,n-2})
+
2^{n/\alpha_2}d_2(t,T_{2,n-2}).
\end{align*}
For every cell $A\in\mathcal A_n$, choose one point $r_n(A)$ such that
\begin{align*}
r_n(A)\in\operatorname*{argmin}_{s\in A}R_n(s),
\end{align*}
and define $\pi_n(t):=r_n(\mathcal A_n(t))$.
Thus, $\pi_n$ is constant on every cell of $\mathcal A_n$.

Since $t$ and $\pi_n(t)$ belong to the same cell, they have the same
nearest-point label at level $n-2$. Therefore, we have that
\begin{align*}
p_{j,n-2}(t)
=
p_{j,n-2}(\pi_n(t)),
\qquad j\in\cbr{1,2}.
\end{align*}
By the triangle inequality, we obtain
\begin{align*}
d_j(t,\pi_n(t))\le
d_j(t,p_{j,n-2}(t))
+
d_j(p_{j,n-2}(\pi_n(t)),\pi_n(t))=
d_j(t,T_{j,n-2})
+
d_j(\pi_n(t),T_{j,n-2}).
\end{align*}
Multiplying by $2^{n/\alpha_j}$ and summing over
$j\in\cbr{1,2}$, we obtain
\begin{align*}
&2^{n/\alpha_1}d_1(t,\pi_n(t))
+
2^{n/\alpha_2}d_2(t,\pi_n(t))\le
R_n(t)+R_n(\pi_n(t)).
\end{align*}
Since $\pi_n(t)$ is the minimal distance, we have
\begin{align*}
R_n(\pi_n(t))
\le
R_n(t).
\end{align*}
Consequently,
\begin{align}
2^{n/\alpha_1}d_1(t,\pi_n(t))
+
2^{n/\alpha_2}d_2(t,\pi_n(t))
\le
2R_n(t).
\label{eq:mixed_common_representative_bound}
\end{align}

Define $a_{\alpha_1,\alpha_2}
:=
\max\cbr{2^{1/\alpha_1},2^{1/\alpha_2}}$.
For every $n\ge3$, the triangle inequality gives
\begin{align*}
d_j(\pi_n(t),\pi_{n-1}(t))
\le
d_j(\pi_n(t),t)
+
d_j(t,\pi_{n-1}(t)).
\end{align*}
Using \eqref{eq:mixed_common_representative_bound} at level $n$, we
obtain $\sum_{j=1}^2
2^{n/\alpha_j}d_j(\pi_n(t),t)
\le
2R_n(t)$.

At level $n-1$, we use $2^{n/\alpha_j}
\le
a_{\alpha_1,\alpha_2}
2^{(n-1)/\alpha_j}$
and obtain that
\begin{align*}
&\sum_{j=1}^2
2^{n/\alpha_j}d_j(t,\pi_{n-1}(t))\le
a_{\alpha_1,\alpha_2}
\sum_{j=1}^2
2^{(n-1)/\alpha_j}d_j(t,\pi_{n-1}(t))\le
2a_{\alpha_1,\alpha_2}R_{n-1}(t).
\end{align*}
Therefore, combining the inequalities above, we have
\begin{align}
&2^{n/\alpha_1}
d_1(\pi_n(t),\pi_{n-1}(t))
+
2^{n/\alpha_2}
d_2(\pi_n(t),\pi_{n-1}(t))\le
2R_n(t)
+
2a_{\alpha_1,\alpha_2}R_{n-1}(t).
\label{eq:mixed_common_parent_edge_bound}
\end{align}

Because the nearest-point maps are eventually the identity, the
partitions $\mathcal A_n$ are eventually the singleton partition of
$K$. Hence, for every $t\in K$, we have $\pi_n(t)=t$ for all sufficiently large $n$.

For every $n\ge3$, define the set of distinct parent-child edges by
\begin{align*}
E_n
:=
\cbr{
(\pi_n(t),\pi_{n-1}(t)):
t\in K
}.
\end{align*}
Since $\mathcal A_n$ refines $\mathcal A_{n-1}$, every cell of
$\mathcal A_n$ is contained in a unique cell of
$\mathcal A_{n-1}$. Both $\pi_n$ and $\pi_{n-1}$ are constant on a
cell of $\mathcal A_n$. Thus, every cell of $\mathcal A_n$ determines
at most one edge in $E_n$, and
\begin{align}
|E_n|
\le
|\mathcal A_n|
\le
2^{2^n}.
\label{eq:mixed_common_edge_cardinality}
\end{align}

We first suppose that $u\ge4$ and define $\ell:=\floor{\log_2u}$. Then, we have
\begin{align*}
\ell\ge2,
\ 
2^\ell\le u<2^{\ell+1}.
\end{align*}
We now apply the mixed-tail increment bound with parameter $16u$ to all
distinct coarse edges $\cbr{(\pi_\ell(t),t_0),
\ t\in K}$. The number of distinct points $\pi_\ell(t)$ is at most
$|\mathcal A_\ell|$. Hence, by a union bound, we have
\begin{align*}
&\PP\left\{
\exists t\in K\ \text{s.t.} \|Z_{\pi_\ell(t)}-Z_{t_0}\|
>
(16u)^{1/\alpha_1}
d_1(\pi_\ell(t),t_0)
+
(16u)^{1/\alpha_2}
d_2(\pi_\ell(t),t_0)
\right\}\\
\le&2|\mathcal A_\ell|e^{-16u}\le
2\exp\cbr{(\log2)2^\ell-16u}\le
2e^{-15u}.
\end{align*}
For every $n>\ell$, apply the mixed-tail increment bound with parameter
$16\cdot2^n$ to every edge in $E_n$. By
\eqref{eq:mixed_common_edge_cardinality} and another union bound,
\begin{align*}
&\PP\left\{\exists (s,r)\in E_n\ \text{s.t.}\|Z_s-Z_r\|
>
(16\cdot2^n)^{1/\alpha_1}d_1(s,r)
+
(16\cdot2^n)^{1/\alpha_2}d_2(s,r)
\right\}\\
&\le
2|E_n|e^{-16\cdot2^n}\le
2\exp\cbr{-(16-\log2)2^n}\le
2e^{-15\cdot2^n}.
\end{align*}
Since $2^{\ell+1}>u$, by direct algebra, we have
\begin{align*}
\sum_{n>\ell}2e^{-15\cdot2^n}\le
2\sum_{r\ge1}e^{-15u2^{r-1}}\le
2\sum_{r\ge1}e^{-15ur}=
\frac{2e^{-15u}}{1-e^{-15u}}.
\end{align*}
For $u\ge4$,  we have that $2e^{-15u}
\le
\frac12e^{-u},
\ 
\frac{2e^{-15u}}{1-e^{-15u}}
\le
\frac12e^{-u}$. 

Therefore, outside an event $\cE_u$ of probability at most $e^{-u}$,
all the preceding coarse and fine increment bounds hold
simultaneously. We now work conditioning on this event $\cE_u$. 

Since $\pi_n(t)=t$ for all sufficiently
large $n$ and $Z_{t_0}=0$, we may telescope and obtain, simultaneously
for every $t\in K$,
\begin{align*}
\|Z_t\|
&\le
\|Z_{\pi_\ell(t)}-Z_{t_0}\|
+
\sum_{n>\ell}
\|Z_{\pi_n(t)}-Z_{\pi_{n-1}(t)}\|\\
&\le
\sum_{j=1}^2
(16u)^{1/\alpha_j}
d_j(\pi_\ell(t),t_0)+
\sum_{n>\ell}
\sum_{j=1}^2
(16\cdot2^n)^{1/\alpha_j}
d_j(\pi_n(t),\pi_{n-1}(t)).
\end{align*}

We bound these two terms separately. For the first term, by the
triangle inequality, we have
\begin{align*}
d_j(\pi_\ell(t),t_0)\le
d_j(\pi_\ell(t),t)+d_j(t,t_0)=
d_j(t,\pi_\ell(t))+v_j(t).
\end{align*}
Define $c_{\alpha_1,\alpha_2}
:=
\max\cbr{16^{1/\alpha_1},16^{1/\alpha_2}}$. Since $u<2^{\ell+1}$, we have
\begin{align*}
u^{1/\alpha_j}
<
2^{1/\alpha_j}2^{\ell/\alpha_j}
\le
a_{\alpha_1,\alpha_2}2^{\ell/\alpha_j}.
\end{align*}
It follows from
\eqref{eq:mixed_common_representative_bound} that
\begin{align*}
\sum_{j=1}^2
(16u)^{1/\alpha_j}
d_j(t,\pi_\ell(t))\le
c_{\alpha_1,\alpha_2}
a_{\alpha_1,\alpha_2}
\sum_{j=1}^2
2^{\ell/\alpha_j}d_j(t,\pi_\ell(t))\le
2c_{\alpha_1,\alpha_2}^{(16)}
a_{\alpha_1,\alpha_2}R_\ell(t).
\end{align*}
The anchor part satisfies that
\begin{align*}\sum_{j=1}^2
(16u)^{1/\alpha_j}v_j(t)\le
c_{\alpha_1,\alpha_2}
\cbr{
u^{1/\alpha_1}v_1(t)
+
u^{1/\alpha_2}v_2(t)
}.
\end{align*}

For the term $\sum_{n>\ell}
\sum_{j=1}^2
(16\cdot2^n)^{1/\alpha_j}
d_j(\pi_n(t),\pi_{n-1}(t))$, by inequality
\eqref{eq:mixed_common_parent_edge_bound}, we have that
\begin{align*}
\sum_{n>\ell}
\sum_{j=1}^2
(16\cdot2^n)^{1/\alpha_j}
d_j(\pi_n(t),\pi_{n-1}(t))\le&
c_{\alpha_1,\alpha_2}
\sum_{n>\ell}
\cbr{
2R_n(t)
+
2a_{\alpha_1,\alpha_2}R_{n-1}(t)
}\\
\le&
2c_{\alpha_1,\alpha_2}
\rbr{1+a_{\alpha_1,\alpha_2}}
\sum_{n\ge2}R_n(t).
\end{align*}
Combining these inequalities together, we obtain
\begin{align}
\|Z_t\|
\le
C_{\alpha_1,\alpha_2}^{(2)}
\cbr{
u^{1/\alpha_1}v_1(t)
+
u^{1/\alpha_2}v_2(t)
+
\sum_{n\ge2}R_n(t)
}.
\label{eq:mixed_finite_pointwise_common_chain}
\end{align}

Finally, by the definition of $R_n(t)$ and the change of variables
$k=n-2$, we have the exact identity
\begin{align*}
\sum_{n\ge2}R_n(t)=
\sum_{n\ge2}
\sum_{j=1}^2
2^{n/\alpha_j}d_j(t,T_{j,n-2})=
\sum_{j=1}^2
2^{2/\alpha_j}
\sum_{k\ge0}
2^{k/\alpha_j}d_j(t,T_{j,k})=
2^{2/\alpha_1}G_1(t)
+
2^{2/\alpha_2}G_2(t).
\end{align*}
Inserting this identity into
\eqref{eq:mixed_finite_pointwise_common_chain} and enlarging the
constant proves the conclusion for $u\ge4$.

We now suppose that $1\le u<4$. Apply the result already proved with
the parameter $4$. Its failure probability is $e^{-4}\le e^{-u}$.
Moreover, since $u\ge1$, by some algebra, we have
\begin{align*}
4^{1/\alpha_j}v_j(t)
\le
4^{1/\alpha_j}
u^{1/\alpha_j}v_j(t),
\ j\in\cbr{1,2}.
\end{align*}
Absorbing the constants $4^{1/\alpha_1}$ and
$4^{1/\alpha_2}$ into $A_{\alpha_1,\alpha_2}$ proves the result for
every $u\ge1$.

THus, we finish the proof.
\end{proof}

Combining the deterministic measure chains with the preceding
probabilistic lemma gives the finite ambient-measure estimate used in
the proof of the theorem.

\begin{lemma}
\label{lem:mixed_tail_finite_ambient_pointwise}
There exists a finite constant $B_{\alpha_1,\alpha_2}$, depending only
on $\alpha_1$ and $\alpha_2$, such that, for every finite set
$K\subseteq T$ containing $t_0$ and every $u\ge1$, with probability at
least $1-e^{-u}$, simultaneously for every $t\in K$, we have
\begin{align*}
\|Z_t\|
\le
B_{\alpha_1,\alpha_2}
\cbr{
\Phi_1(t)+\Phi_2(t)
+
u^{1/\alpha_1}v_1(t)
+
u^{1/\alpha_2}v_2(t)
}.
\end{align*}
\end{lemma}

\begin{proof}[Proof of Lemma
\ref{lem:mixed_tail_finite_ambient_pointwise}]
By Lemma \ref{lem:mixed_tail_pointwise_measure_chain}, for every
$j\in\cbr{1,2}$, there exists an admissible sequence
$\mathcal T_j=(T_{j,k})_{k\ge0}$, anchored at $t_0$ and eventually
equal to $K$, such that, simultaneously for every $t\in K$,
\begin{align*}
G_j(t)
:=
\sum_{k\ge0}
2^{k/\alpha_j}d_j(t,T_{j,k})
\le
C_{\alpha_j}^{(1)}
\rbr{\Phi_j(t)+v_j(t)}.
\end{align*}
Apply Lemma \ref{lem:mixed_tail_finite_pointwise} to these two
sequences. Outside an event of probability at most $e^{-u}$, for every
$t\in K$, we have
\begin{align*}
\|Z_t\|
&\le
A_{\alpha_1,\alpha_2}
\cbr{
G_1(t)+G_2(t)
+
u^{1/\alpha_1}v_1(t)
+
u^{1/\alpha_2}v_2(t)
}\\
&\le
A_{\alpha_1,\alpha_2}
\left[
C_{\alpha_1}^{(1)}\Phi_1(t)
+
C_{\alpha_2}^{(1)}\Phi_2(t)+
C_{\alpha_1}^{(1)}v_1(t)
+
C_{\alpha_2}^{(1)}v_2(t)+
u^{1/\alpha_1}v_1(t)
+
u^{1/\alpha_2}v_2(t)
\right].
\end{align*}
Since $u\ge1$, we have
\begin{align*}
v_j(t)
\le
u^{1/\alpha_j}v_j(t),
\qquad j\in\cbr{1,2}.
\end{align*}
Therefore, we have
\begin{align*}
C_{\alpha_j}^{(1)}v_j(t)
+
u^{1/\alpha_j}v_j(t)
\le
\rbr{C_{\alpha_j}^{(1)}+1}
u^{1/\alpha_j}v_j(t).
\end{align*}
Absorbing the finitely many constants into
$B_{\alpha_1,\alpha_2}$, we finish the proof.
\end{proof}

We are now ready to prove the main theorem.

\begin{proof}[Proof of Theorem
\ref{thm:mixed_tail_majorization}] For positive rational vectors
$\boldsymbol q=(q_1,q_2)\in\mathbb Q_{>0}^2,
\ 
\boldsymbol s=(s_1,s_2)\in\mathbb Q_{>0}^2$,
we define
\begin{align*}
H_{\boldsymbol q,\boldsymbol s}
:=
\cbr{
t\in T:
\Phi_j(t)\le q_j
\text{ and }
v_j(t)\le s_j
\text{ for }j=1,2
}.
\end{align*}
By Lemma \ref{lem:mixed_tail_measurability}, every set
$H_{\boldsymbol q,\boldsymbol s}$ is Borel measurable.

We next choose countable dense subsets of these sets. Since
$(T,\rho)$ is a separable pseudo-metric space, its metric quotient by
the relation $\rho(s,t)=0$ is a separable metric space and hence is
second countable. $H_{\boldsymbol q,\boldsymbol s}$ is a subset of $T$. Consequently, every
$H_{\boldsymbol q,\boldsymbol s}$ is separable with respect to $\rho$.

For every pair
$(\boldsymbol q,\boldsymbol s)$, we choose a countable $\rho$-dense subset such that
$D_{\boldsymbol q,\boldsymbol s}
\subseteq
H_{\boldsymbol q,\boldsymbol s}$. If the corresponding set is empty, take
$D_{\boldsymbol q,\boldsymbol s}=\varnothing$.

The collection of positive rational vectors is countable. Hence, the
union
\begin{align*}
D
:=
\cbr{t_0}\cup
\bigcup_{
\boldsymbol q\in\mathbb Q_{>0}^2,\,
\boldsymbol s\in\mathbb Q_{>0}^2
}
D_{\boldsymbol q,\boldsymbol s}
\end{align*}
is countable. Beginning with the anchor, we enumerate it as $D=\cbr{t_0,t_1,t_2,\ldots}$.

If $D$ is finite, repeat $t_0$ in the enumeration so that the same
notation remains valid. Define the increasing sequence of finite sets
\begin{align*}
K_N:=\cbr{t_0,t_1,\ldots,t_N},
\ N\ge1.
\end{align*}
For every $N\ge1$, let $O_N$ be the event
\begin{align*}
O_N
:=
\bigcap_{t\in K_N}
\cbr{
\|Z_t\|
\le
B_{\alpha_1,\alpha_2}
\left[
\Phi_1(t)+\Phi_2(t)
+
u^{1/\alpha_1}v_1(t)
+
u^{1/\alpha_2}v_2(t)
\right]
}.
\end{align*}
These are the actual conclusion events in
Lemma \ref{lem:mixed_tail_finite_ambient_pointwise}. Then, we have that
\begin{align*}
\PP(O_N)\ge1-e^{-u},
\qquad N\ge1.
\end{align*}
Moreover, since $K_N\subseteq K_{N+1}$, the events are decreasing, i.e., $
E_{N+1}\subseteq E_N$.

By the continuity of probability, we have that,
\begin{align*}
\PP\left(
\bigcap_{N\ge1}O_N
\right)
&=
\lim_{N\to\infty}\PP(O_N)\ge
1-e^{-u}.
\end{align*}

Fix an outcome in this event. By the null-set modification made at the
beginning of the appendix, the sample path $t\mapsto Z_t$ is
$\rho$-continuous. Let $t\in T$ satisfy $\Phi_1(t)<\infty,
\ 
\Phi_2(t)<\infty$ and choose arbitrary positive rational numbers such that
\begin{align*}
q_j>\Phi_j(t),
\ 
s_j>v_j(t),
\  j\in\cbr{1,2}.
\end{align*}
Then, we have $t\in H_{\boldsymbol q,\boldsymbol s}$.
Since $D_{\boldsymbol q,\boldsymbol s}$ is $\rho$-dense in this set,
there exists a sequence $\cbr{x_m}_{m\ge1}
\subseteq
D_{\boldsymbol q,\boldsymbol s}$ such that $\rho(x_m,t)\rightarrow0$.

Every point $x_m$ belongs to $D$, and hence it belongs to $K_N$ for
all sufficiently large $N$. Conditioned on the event
$\bigcap_{N\ge1}O_N$, we have that
\begin{align*}
\|Z_{x_m}\|
&\le
B_{\alpha_1,\alpha_2}
\cbr{
\Phi_1(x_m)+\Phi_2(x_m)
+
u^{1/\alpha_1}v_1(x_m)
+
u^{1/\alpha_2}v_2(x_m)
}\\
&\le
B_{\alpha_1,\alpha_2}
\cbr{
q_1+q_2
+
u^{1/\alpha_1}s_1
+
u^{1/\alpha_2}s_2
}.
\end{align*}
Recall the path continuity and we have
$Z_{x_m}\rightarrow Z_t
\ 
\text{in }(\mathbb B,\|\cdot\|)$.

Hence, we get
\begin{align*}
\|Z_t\|
&=
\lim_{m\to\infty}\|Z_{x_m}\|\le
B_{\alpha_1,\alpha_2}
\cbr{
q_1+q_2
+
u^{1/\alpha_1}s_1
+
u^{1/\alpha_2}s_2
}.
\end{align*}
Letting $q_j\downarrow\Phi_j(t)$ and
$s_j\downarrow v_j(t)$ through rational values and we obtain
\begin{align*}
\|Z_t\|
\le
B_{\alpha_1,\alpha_2}
\cbr{
\Phi_1(t)+\Phi_2(t)
+
u^{1/\alpha_1}v_1(t)
+
u^{1/\alpha_2}v_2(t)
}.
\end{align*}
The event used above does not depend on $t$. Hence, the inequality
holds simultaneously for every $t\in T$ for which both profiles are
finite. If either profile is infinite, the assertion is vacuous.

Finally, set $u:=\log\frac{e}{\delta}$.
Then $u\ge1$ and $e^{-u}
=
\frac{\delta}{e}
\le
\delta.$
The theorem follows with
$C_{\alpha_1,\alpha_2}:=B_{\alpha_1,\alpha_2}$.
Therefore, with probability at least $1-\delta$, simultaneously for
every $t\in T$,
\begin{align*}
\|Z_t\|
\le
C_{\alpha_1,\alpha_2}
\cbr{
\Phi_1(t)+\Phi_2(t)
+
v_1(t)
\left(
\log\frac{e}{\delta}
\right)^{1/\alpha_1}
+
v_2(t)
\left(
\log\frac{e}{\delta}
\right)^{1/\alpha_2}
}.
\end{align*}
We finish the proof.
\end{proof}
\begin{proof}[Proof of Corollary \ref{cor:sub-weibull}]
    The proof is intermediate. We can just set $d_2\equiv 0$ to be the zero metric and apply Theorem \ref{thm:mixed_tail_majorization}. Then, we shall obtain the result.
\end{proof}

\section{Proofs in Section \ref{sec:multi_regime_majorization}}
\label{app:proofs_multi_regime}

Throughout this appendix, we use the convention $\log\frac{1}{0}:=+\infty$. Changing the process on a null event does not affect any of the claims
below. We therefore redefine $Z_t(\omega):=0,
\ \omega\notin\Omega_0,\ t\in T$ as in Appendix \ref{app:proofs_mixed_tail}.
Thus, every sample path used below is $\rho$-continuous and satisfies
$Z_{t_0}=0$.

We first verify the measurability of the pointwise radii and local
functionals.

\begin{lemma}
\label{lem:multi_regime_measurability}
For every $j\in\cbr{1,\ldots,m}$, the map
\begin{align*}
(t,r)
\longmapsto
\mu_j(B_{d_j}(t,r))
\end{align*}
is Borel measurable on $T\times[0,\infty)$. Consequently, $v_j$ and
$\Phi_j$ are $[0,\infty]$-valued Borel functions on $T$.
\end{lemma}

\begin{proof}[Proof of Lemma \ref{lem:multi_regime_measurability}]
Fix $j\in\cbr{1,\ldots,m}$. Since $d_j$ is jointly
$\mathcal A\otimes\mathcal A$-measurable, the map
\begin{align*}
(t,r,s)
\longmapsto
\mathbf 1_{\cbr{d_j(t,s)\le r}}
\end{align*}
is measurable on
$T\times[0,\infty)\times T$. Therefore, by measurability of
parameterized integrals, we have that
\begin{align*}
(t,r)
\longmapsto
\int_T
\mathbf 1_{\cbr{d_j(t,s)\le r}}\mu_j(ds)
=
\mu_j(B_{d_j}(t,r))
\end{align*}
is Borel measurable. Since $t_0$ is fixed, the map
$v_j(t)=d_j(t,t_0)$ is also Borel measurable.

Define
\begin{align*}
Q_j(t,r)
:=
\begin{cases}
\displaystyle
\left(
\log\frac{1}{\mu_j(B_{d_j}(t,r))}
\right)^{1/\alpha_j},
&0\le r\le4v_j(t),\\[1.2ex]
0,
&r>4v_j(t).
\end{cases}
\end{align*}
The preceding measurability statements show that $Q_j$ is a
nonnegative, extended-real-valued Borel function. The piecewise
definition avoids the ambiguous product $0\cdot\infty$ when a ball has
zero $\mu_j$-mass. By the definition of $\Phi_j$, we have
\begin{align*}
\Phi_j(t)
=
\int_0^\infty Q_j(t,r)dr.
\end{align*}
Another application of measurability of parameterized integrals proves
that $\Phi_j$ is Borel measurable. The same argument applies to every
$j\in\cbr{1,\ldots,m}$, which finishes the proof.
\end{proof}

We next transfer each fixed ambient measure to an arbitrary finite
subset. This nearest-point construction permits all the measure-generated
chains to be built inside the same finite set while retaining the ambient
ball masses in their respective pseudo-metrics.

\begin{lemma}
\label{lem:multi_regime_ambient_restriction}
Let $K\subseteq T$ be finite and contain $t_0$. For
$j\in\cbr{1,\ldots,m}$, write
$d_{j,K}:=d_j|_{K\times K}$. There exists a probability measure
$\widetilde\mu_{j,K}$ on $K$ such that, for every $t\in K$ and every
$r\ge0$,
\begin{align*}
\widetilde\mu_{j,K}(\cbr{t_0})
\ge\frac12,
\qquad
\widetilde\mu_{j,K}(B_{d_{j,K}}(t,2r))
\ge
\frac12\mu_j(B_{d_j}(t,r)).
\end{align*}
Moreover, if $c_{\alpha_j}^{(0)}:=
\max\cbr{1,2^{1/\alpha_j-1}}$,
then we have
\begin{align}
&\int_0^{4v_j(t)}
\left(
\log\frac{1}{
\widetilde\mu_{j,K}(B_{d_{j,K}}(t,r))
}
\right)^{1/\alpha_j}dr\le
2c_{\alpha_j}^{(0)}\Phi_j(t)
+
4c_{\alpha_j}^{(0)}
(\log2)^{1/\alpha_j}v_j(t).
\label{eq:multi_regime_ambient_integral_comparison}
\end{align}
\end{lemma}

\begin{proof}[Proof of Lemma
\ref{lem:multi_regime_ambient_restriction}]
Fix any $j\in\cbr{1,\ldots,m}$ and enumerate $K=\cbr{k_1,\ldots,k_N}$.
For every $s\in T$, let $q_{j,K}(s)$ be a point in $K$ minimizing
$d_j(s,\cdot)$, with ties broken by choosing the point having the
smallest index, i.e., $q_{j,K}$ is the nearest point map. More explicitly, define
\begin{align*}
C_i
:={}&
\bigcap_{r<i}
\cbr{s\in T:d_j(s,k_i)<d_j(s,k_r)}
\cap
\bigcap_{r>i}
\cbr{s\in T:d_j(s,k_i)\le d_j(s,k_r)}.
\end{align*}
Every $C_i$ is Borel measurable because $d_j$ is jointly Borel. The
sets $C_1,\ldots,C_N$ form a partition of $T$, and
$q_{j,K}(s)=k_i$ for $s\in C_i$. Hence, $q_{j,K}:T\to K$ is Borel
measurable and satisfies
\begin{align*}
d_j(s,q_{j,K}(s))
=
d_j(s,K)
:=
\min_{k\in K}d_j(s,k).
\end{align*}

We define
\begin{align*}
\widetilde\mu_{j,K}
:=
\frac12\delta_{t_0}
+
\frac12(q_{j,K})_\sharp\mu_j,
\end{align*}
where$
(q_{j,K})_\sharp\mu_j(A)
:=
\mu_j(q_{j,K}^{-1}(A)),
\ A\subseteq K$.
By construction, we have $\widetilde\mu_{j,K}(\cbr{t_0})\ge\frac12$.

We now prove the ball-mass comparison. Fix $t\in K$ and $r\ge0$.
If $s\in B_{d_j}(t,r)$, then the nearest-point property and the fact
that $t\in K$ imply that
\begin{align*}
d_j(s,q_{j,K}(s))
\le
d_j(s,t)
\le r.
\end{align*}
By the triangle inequality, we have
\begin{align*}
d_j(t,q_{j,K}(s))
&\le
d_j(t,s)+d_j(s,q_{j,K}(s))
\le2r.
\end{align*}
Therefore, we obtain $q_{j,K}(B_{d_j}(t,r))
\subseteq
B_{d_{j,K}}(t,2r)$.
It follows that
\begin{align*}
\widetilde\mu_{j,K}(B_{d_{j,K}}(t,2r))
&\ge
\frac12(q_{j,K})_\sharp\mu_j(B_{d_{j,K}}(t,2r))=
\frac12\mu_j(q_{j,K}^{-1}(B_{d_{j,K}}(t,2r)))
\ge
\frac12\mu_j(B_{d_j}(t,r)).
\end{align*}
We now prove the integral comparison. Put
$p_j:=1/\alpha_j$. By the change of variables $r=2u$, we have
\begin{align*}
&\int_0^{4v_j(t)}
\left(
\log\frac{1}{
\widetilde\mu_{j,K}(B_{d_{j,K}}(t,r))
}
\right)^{p_j}dr=
2\int_0^{2v_j(t)}
\left(
\log\frac{1}{
\widetilde\mu_{j,K}(B_{d_{j,K}}(t,2u))
}
\right)^{p_j}du.
\end{align*}
The ball-mass comparison gives
\begin{align*}
\log\frac{1}{
\widetilde\mu_{j,K}(B_{d_{j,K}}(t,2u))
}
&\le
\log\frac{2}{\mu_j(B_{d_j}(t,u))}=
\log\frac{1}{\mu_j(B_{d_j}(t,u))}
+
\log2.
\end{align*}
For every $x,y\ge0$, recall the definition of $c_{\alpha_j}^{(0)}$, by direct algebra, we have
$(x+y)^{p_j}
\le
c_{\alpha_j}^{(0)}
\rbr{x^{p_j}+y^{p_j}}$.
Combining the preceding inequalities, we obtain
\begin{align*}
\int_0^{4v_j(t)}
\left(
\log\frac{1}{
\widetilde\mu_{j,K}(B_{d_{j,K}}(t,r))
}
\right)^{p_j}dr&\le
2c_{\alpha_j}^{(0)}
\int_0^{2v_j(t)}
\left(
\log\frac{1}{\mu_j(B_{d_j}(t,u))}
\right)^{p_j}du
+
4c_{\alpha_j}^{(0)}
(\log2)^{p_j}v_j(t)\\
&\le
2c_{\alpha_j}^{(0)}\Phi_j(t)
+
4c_{\alpha_j}^{(0)}
(\log2)^{p_j}v_j(t),
\end{align*}
where we used $2v_j(t)\le4v_j(t)$ in the final inequality. This
proves the assertion for the fixed $j$. Applying the same construction
to every $j\in[m]$ finishes the proof.
\end{proof}

We next construct, for each pseudo-metric, an admissible sequence whose
chaining cost remains pointwise rather than being replaced by a worst-case
supremum over $K$.

\begin{lemma}
\label{lem:multi_regime_pointwise_measure_chain}
Let $K\subseteq T$ be finite and contain $t_0$. For every
$j\in\cbr{1,\ldots,m}$, there exists an admissible sequence
$\mathcal T_j=(T_{j,k})_{k\ge0}$ of subsets of $K$ such that
$T_{j,0}=\cbr{t_0}$ and $T_{j,k}=K$ for all sufficiently large $k$.
Moreover, for every $t\in K$,
\begin{align}
G_{j,K}(t)
:=
\sum_{k\ge0}2^{k/\alpha_j}d_j(t,T_{j,k})
\le
C_{\alpha_j}^{(1)}
\rbr{\Phi_j(t)+v_j(t)},
\label{eq:multi_regime_metric_chain_cost}
\end{align}
where $C_{\alpha_j}^{(1)}<\infty$ depends only on $\alpha_j$.
\end{lemma}

\begin{proof}[Proof of Lemma
\ref{lem:multi_regime_pointwise_measure_chain}]
Fix $j\in\cbr{1,\ldots,m}$, put
$p_j:=1/\alpha_j$, and let $\widetilde\mu_{j,K}$ be the probability
measure supplied by Lemma \ref{lem:multi_regime_ambient_restriction}.
For $t\in K$ and $n\ge0$, define
\begin{align*}
r_{j,n}(t)
:=
\min\cbr{
r\in\cbr{d_j(t,s):s\in K}:
\widetilde\mu_{j,K}(B_{d_{j,K}}(t,r))
\ge2^{-2^n}
}.
\end{align*}
The minimum exists because $K$ is finite and the largest
$d_{j,K}$-ball centered at $t$ has mass one. Since the mass threshold
$2^{-2^n}$ is nonincreasing in $n$, the sequence
$(r_{j,n}(t))_{n\ge0}$ is nonincreasing.

For a fixed $n\ge0$, we order the points of $K$ by nondecreasing
$r_{j,n}(t)$, breaking ties according to an arbitrary fixed enumeration
of $K$. Traverse this ordered list and retain a point $t$ precisely when
\begin{align*}
d_j(t,s)
>
r_{j,n}(t)+r_{j,n}(s)
\end{align*}
for every previously retained point $s$. Denote the set of retained
points by $S_{j,n}$.

 Similar to Appendix \ref{app:proofs_mixed_tail}, we first prove that the balls
$\cbr{
B_{d_{j,K}}(s,r_{j,n}(s)):
s\in S_{j,n}
}$
are pairwise disjoint. Suppose that two such balls, centered at distinct
points $s,t\in S_{j,n}$, intersect. Then there exists $z\in K$ such
that
$d_j(s,z)\le r_{j,n}(s),
\ 
d_j(t,z)\le r_{j,n}(t)$.

The triangle inequality would then give
\begin{align*}
d_j(s,t)
\le
d_j(s,z)+d_j(z,t)
\le
r_{j,n}(s)+r_{j,n}(t),
\end{align*}
 which contradicts the greedy retention rule.

By the definition of $r_{j,n}(s)$, every one of these disjoint balls
has $\widetilde\mu_{j,K}$-mass at least $2^{-2^n}$. Hence,
\begin{align*}
1
\ge
\sum_{s\in S_{j,n}}
\widetilde\mu_{j,K}(B_{d_{j,K}}(s,r_{j,n}(s)))
\ge
|S_{j,n}|2^{-2^n}.
\end{align*}
Thus, we have $|S_{j,n}|
\le
2^{2^n}$.
We next prove that $S_{j,n}$ approximates every point at its own mass
radius. If $t\in S_{j,n}$, then $d_j(t,S_{j,n})=0$. If
$t\notin S_{j,n}$, the point $t$ was rejected because there exists a
previously retained point $s\in S_{j,n}$ such that
\begin{align*}
r_{j,n}(s)\le r_{j,n}(t),
\ 
d_j(t,s)
\le
r_{j,n}(t)+r_{j,n}(s).
\end{align*}
Therefore, in both cases, we have
\begin{align}
d_j(t,S_{j,n})
\le
2r_{j,n}(t),
\ t\in K.
\label{eq:multi_regime_mass_net_approximation}
\end{align}

For $t\in K$, define
\begin{align*}
F_{j,t}(r)
:=
\left(
\log\frac{1}{
\widetilde\mu_{j,K}(B_{d_{j,K}}(t,r))
}
\right)^{p_j}.
\end{align*}
If $r<r_{j,n}(t)$, by the definition of $r_{j,n}(t)$, we have
$\widetilde\mu_{j,K}(B_{d_{j,K}}(t,r))
<
2^{-2^n}$ and hence
\begin{align}
F_{j,t}(r)
\ge
(\log2)^{p_j}2^{np_j}.
\label{eq:multi_regime_mass_integrand_lower_bound}
\end{align}
Since the intervals $[r_{j,n+1}(t),r_{j,n}(t)),
\ n\ge0$ are pairwise disjoint and contained in $[0,r_{j,0}(t)]$, inequality
\eqref{eq:multi_regime_mass_integrand_lower_bound} gives
\begin{align}
\int_0^{r_{j,0}(t)}F_{j,t}(r)dr
&\ge
(\log2)^{p_j}
\sum_{n\ge0}2^{np_j}
\rbr{r_{j,n}(t)-r_{j,n+1}(t)}.
\label{eq:multi_regime_mass_integral_radius_differences}
\end{align}

We compare the last sum with the weighted sum of the mass radii. If the
integral on the left-hand side is infinite, the desired estimate is
immediate. Suppose that it is finite. Then we have $\widetilde\mu_{j,K}(B_{d_{j,K}}(t,0))>0$.
Indeed, if this mass were zero, the finiteness of $K$ would imply that
$F_{j,t}$ is infinite on an interval of positive length. Since
$2^{-2^n}\downarrow0$, the positive mass at radius zero implies that
$r_{j,n}(t)=0$ for all sufficiently large $n$. Summation by parts is
therefore justified and yields
\begin{align*}
\sum_{n\ge0}2^{np_j}
\rbr{r_{j,n}(t)-r_{j,n+1}(t)}=
\rbr{1-2^{-p_j}}
\sum_{n\ge0}2^{np_j}r_{j,n}(t)
+
2^{-p_j}r_{j,0}(t)\ge
\rbr{1-2^{-p_j}}
\sum_{n\ge0}2^{np_j}r_{j,n}(t).
\end{align*}
Combining this identity with
\eqref{eq:multi_regime_mass_integral_radius_differences}, we obtain
\begin{align}
\sum_{n\ge0}2^{np_j}r_{j,n}(t)
\le
\frac{1}{
(\log2)^{p_j}(1-2^{-p_j})
}
\int_0^{r_{j,0}(t)}F_{j,t}(r)dr.
\label{eq:multi_regime_weighted_mass_radius_bound}
\end{align}
We now define the admissible sequence. Let 
\begin{align*}
T_{j,0}:=\cbr{t_0},
\ 
T_{j,k}:=S_{j,k-1},
\ k\ge1.
\end{align*}
By the property of $S_{j,k}$, we know that $|T_{j,k}|\le
2^{2^{k-1}}\le
2^{2^k}$,
so the sequence is admissible. Since $K$ is finite, choose
$k_K\ge1$ such that $2^{2^{k_K}}\ge|K|$, and replace $T_{j,k}$ by
$K$ for every $k\ge k_K$. This replacement only decreases
$d_j(t,T_{j,k})$, preserves admissibility, and ensures that the
sequence is eventually equal to $K$.

Since $T_{j,0}=\cbr{t_0}$,  $d_j(t,T_{j,0})
=
v_j(t)$. Moreover, by \eqref{eq:multi_regime_mass_net_approximation}, for every
$k\ge1$, we have
\begin{align*}
d_j(t,T_{j,k})
\le
2r_{j,k-1}(t).
\end{align*}
The same inequality remains valid at levels at which $T_{j,k}$ was
replaced by $K$, because its left-hand side is then zero. Consequently, we obtain
\begin{align*}
G_{j,K}(t)
&=
\sum_{k\ge0}2^{kp_j}d_j(t,T_{j,k})\le
v_j(t)
+
2\sum_{k\ge1}2^{kp_j}r_{j,k-1}(t)=
v_j(t)
+
2^{1+p_j}
\sum_{n\ge0}2^{np_j}r_{j,n}(t).
\end{align*}
Using \eqref{eq:multi_regime_weighted_mass_radius_bound}, we conclude
that
\begin{align}
G_{j,K}(t)
&\le
v_j(t)
+
\kappa_{\alpha_j}
\int_0^{r_{j,0}(t)}F_{j,t}(r)dr,
\label{eq:multi_regime_chain_cost_restricted_integral}
\end{align}
where $\kappa_{\alpha_j}
:=
\frac{2^{1+1/\alpha_j}}{
(\log2)^{1/\alpha_j}(1-2^{-1/\alpha_j})
}$.

The atom at the anchor gives the required pointwise truncation. The
closed ball $B_{d_{j,K}}(t,v_j(t))$ contains $t_0$, and $\widetilde\mu_{j,K}(\cbr{t_0})
\ge
\frac12
=
2^{-2^0}$. Therefore, the definition of $r_{j,0}(t)$ implies
\begin{align*}
r_{j,0}(t)
\le
v_j(t).
\end{align*}
It follows from Lemma \ref{lem:multi_regime_ambient_restriction} that
\begin{align*}
\int_0^{r_{j,0}(t)}F_{j,t}(r)dr
&\le
\int_0^{4v_j(t)}F_{j,t}(r)dr\le
2c_{\alpha_j}^{(0)}\Phi_j(t)
+
4c_{\alpha_j}^{(0)}
(\log2)^{p_j}v_j(t).
\end{align*}
Inserting this estimate into
\eqref{eq:multi_regime_chain_cost_restricted_integral}, we obtain
\begin{align*}
G_{j,K}(t)
&\le
2\kappa_{\alpha_j}c_{\alpha_j}^{(0)}\Phi_j(t)+
\left[
1+4\kappa_{\alpha_j}c_{\alpha_j}^{(0)}(\log2)^{p_j}
\right]v_j(t).
\end{align*}
Thus, \eqref{eq:multi_regime_metric_chain_cost} holds with
\begin{align*}
C_{\alpha_j}^{(1)}
:=
\max\cbr{
2\kappa_{\alpha_j}c_{\alpha_j}^{(0)},
1+4\kappa_{\alpha_j}c_{\alpha_j}^{(0)}(\log2)^{1/\alpha_j}
}.
\end{align*}
The construction is deterministic and the inequality holds
simultaneously for every $t\in K$. Repeating the construction for all
$j\in\cbr{1,\ldots,m}$ and we finish the proof.
\end{proof}

We now combine the $m$ admissible sequences. The common refinement is
essential because the increment condition controls the sum of all metric
contributions for one process and one common deviation parameter.

\begin{lemma}
\label{lem:multi_regime_finite_pointwise}
Let $K\subseteq T$ be finite and contain $t_0$. For every
$j\in\cbr{1,\ldots,m}$, let
$\mathcal T_j=(T_{j,k})_{k\ge0}$ be an admissible sequence of subsets
of $K$ such that $T_{j,0}=\cbr{t_0}$ and $T_{j,k}=K$ for all
sufficiently large $k$. Define
\begin{align*}
G_j(t)
:=
\sum_{k\ge0}2^{k/\alpha_j}d_j(t,T_{j,k}),
\qquad t\in K.
\end{align*}
Let $b_m
:=
1+\ceil{\log_2m}$. There exists a finite constant
$A_{m,\boldsymbol\alpha}$, depending only on $m$ and
$\boldsymbol\alpha$, such that, for every $u\ge1$, with probability
at least $1-e^{-u}$, simultaneously for every $t\in K$,
\begin{align}
\|Z_t\|
\le
A_{m,\boldsymbol\alpha}
\left[
\sum_{j=1}^mG_j(t)
+
\sum_{j=1}^m u^{1/\alpha_j}v_j(t)
\right].
\label{eq:multi_regime_finite_chain_bound}
\end{align}
More explicitly, the proof permits the choice
\begin{align}
A_{m,\boldsymbol\alpha}
:=
2c_{\boldsymbol\alpha}^{(16)}
\rbr{1+2a_{\boldsymbol\alpha}}
M_{m,\boldsymbol\alpha},
\label{eq:multi_regime_explicit_finite_constant}
\end{align}
where
\begin{align*}
a_{\boldsymbol\alpha}
:=
\max_{1\le j\le m}2^{1/\alpha_j},
\qquad
c_{\boldsymbol\alpha}^{(16)}
:=
\max_{1\le j\le m}16^{1/\alpha_j},
\qquad
M_{m,\boldsymbol\alpha}
:=
\max_{1\le j\le m}2^{b_m/\alpha_j}.
\end{align*}
\end{lemma}

\begin{proof}[Proof of Lemma \ref{lem:multi_regime_finite_pointwise}]
For $j\in\cbr{1,\ldots,m}$, $k\ge0$, and $t\in K$, choose a
nearest point $p_{j,k}(t)\in T_{j,k}$ such that
\begin{align*}
d_j(t,p_{j,k}(t))
=
d_j(t,T_{j,k}).
\end{align*}
Such a point exists because $T_{j,k}$ is a nonempty finite set. When
$T_{j,k}=K$, we choose $p_{j,k}(t):=t$.

For every $n\ge b_m$, define an equivalence relation on $K$ by
\begin{align*}
s\sim_n t
\quad\text{if and only if}\quad
p_{j,k}(s)=p_{j,k}(t)
\end{align*}
for every $j\in\cbr{1,\ldots,m}$ and every
$0\le k\le n-b_m$. Let $\mathcal A_n$ be the corresponding
partition of $K$, and write $\mathcal A_n(t)$ for the cell containing
$t$. Equivalently, two points belong to the same cell precisely when
all the entries of $P_n(t)
:=
\rbr{
p_{j,k}(t):
1\le j\le m,\ 0\le k\le n-b_m
}$
agree.

The partitions are nested. Indeed, agreement of all nearest-point labels
through level $n-b_m$ implies agreement through level
$n-1-b_m$. Hence, for every $n\ge b_m+1$,
\begin{align*}
s\sim_n t
\quad\Longrightarrow\quad
s\sim_{n-1}t,
\end{align*}
so $\mathcal A_n$ refines $\mathcal A_{n-1}$.

We next bound the number of cells. By admissibility of the $m$
sequences,
\begin{align*}
|\mathcal A_n|\le
\prod_{j=1}^m
\prod_{k=0}^{n-b_m}|T_{j,k}|\le
\prod_{j=1}^m
\prod_{k=0}^{n-b_m}2^{2^k}=
2^{m\sum_{k=0}^{n-b_m}2^k}=
2^{m(2^{n-b_m+1}-1)}.
\end{align*}
By the definition of $b_m$, we have that $m
\le
2^{b_m-1}$, which leads to
\begin{align}
|\mathcal A_n|
\le
2^{m2^{n-b_m+1}}
\le
2^{2^n}.
\label{eq:multi_regime_A_n_bound}
\end{align}
This is the step for which the level shift $b_m$ is required.

For every $n\ge b_m$, define the combined pointwise approximation cost
\begin{align*}
R_n(t)
:=
\sum_{j=1}^m
2^{n/\alpha_j}d_j(t,T_{j,n-b_m}).
\end{align*}
For every cell $A\in\mathcal A_n$, choose a point
$r_n(A)
\in
\operatorname*{argmin}_{s\in A}R_n(s)$ and define $\pi_n(t)
:=
r_n(\mathcal A_n(t))$.
Thus, $\pi_n$ is constant on every cell of $\mathcal A_n$.
Since $t$ and $\pi_n(t)$ belong to the same cell, for every
$j\in\cbr{1,\ldots,m}$, we have
\begin{align*}
p_{j,n-b_m}(t)
=
p_{j,n-b_m}(\pi_n(t)).
\end{align*}
The triangle inequality therefore gives
\begin{align*}
d_j(t,\pi_n(t))
&\le
d_j(t,p_{j,n-b_m}(t))
+
d_j(p_{j,n-b_m}(\pi_n(t)),\pi_n(t))=
d_j(t,T_{j,n-b_m})
+
d_j(\pi_n(t),T_{j,n-b_m}).
\end{align*}
Multiplying by $2^{n/\alpha_j}$ and summing over all $j$, we obtain
\begin{align*}
\sum_{j=1}^m
2^{n/\alpha_j}d_j(t,\pi_n(t))
\le
R_n(t)+R_n(\pi_n(t)).
\end{align*}
Since $\pi_n(t)$ minimizes the distance, we have $R_n(\pi_n(t))
\le
R_n(t)$ and hence
\begin{align}
\sum_{j=1}^m
2^{n/\alpha_j}d_j(t,\pi_n(t))
\le
2R_n(t).
\label{eq:multi_regime_common_representative_bound}
\end{align}

For every $n\ge b_m+1$, again, applying the triangle inequality and we have
\begin{align*}
d_j(\pi_n(t),\pi_{n-1}(t))
\le
d_j(\pi_n(t),t)
+
d_j(t,\pi_{n-1}(t)).
\end{align*}
We use \eqref{eq:multi_regime_common_representative_bound} at level
$n$ controls the first part. For the second part, notice that
$2^{n/\alpha_j}
\le
a_{\boldsymbol\alpha}2^{(n-1)/\alpha_j}$. Applying \eqref{eq:multi_regime_common_representative_bound} at level
$n-1$, we obtain
\begin{align*}
\sum_{j=1}^m
2^{n/\alpha_j}d_j(t,\pi_{n-1}(t))
&\le
a_{\boldsymbol\alpha}
\sum_{j=1}^m
2^{(n-1)/\alpha_j}d_j(t,\pi_{n-1}(t))\le
2a_{\boldsymbol\alpha}R_{n-1}(t).
\end{align*}
Therefore,
\begin{align}
\sum_{j=1}^m
2^{n/\alpha_j}
d_j(\pi_n(t),\pi_{n-1}(t))
\le
2R_n(t)
+
2a_{\boldsymbol\alpha}R_{n-1}(t).
\label{eq:multi_regime_common_parent_edge_bound}
\end{align}

Because every nearest-point map is eventually the identity, the
partitions $\mathcal A_n$ are eventually the singleton partition of
$K$. Hence, for every $t\in K$, we have
$\pi_n(t)=t$ for all sufficiently large $n$.

For every $n\ge b_m+1$, define the set of distinct parent-child edges
\begin{align*}
E_n
:=
\cbr{
(\pi_n(t),\pi_{n-1}(t)):
t\in K
}.
\end{align*}
Since $\mathcal A_n$ refines $\mathcal A_{n-1}$, every cell of
$\mathcal A_n$ is contained in a unique cell of
$\mathcal A_{n-1}$. Both $\pi_n$ and $\pi_{n-1}$ are constant on a
cell of $\mathcal A_n$. Thus, every cell of $\mathcal A_n$ determines
at most one edge in $E_n$, and
\begin{align}
|E_n|
\le
|\mathcal A_n|
\le
2^{2^n}.
\label{eq:multi_regime_common_edge_cardinality}
\end{align}

We first suppose that $u\ge2^{b_m}$ and define $\ell:=\floor{\log_2u}$.
Then, we have
\begin{align*}
\ell\ge b_m,
\ 
2^\ell\le u<2^{\ell+1}.
\end{align*}
Apply the increment condition with parameter $16u$ to all
distinct coarse edges
\begin{align*}
\cbr{(\pi_\ell(t),t_0):t\in K}.
\end{align*}
The number of distinct points $\pi_\ell(t)$ is at most
$|\mathcal A_\ell|$. Hence, by a union bound and \eqref{eq:multi_regime_A_n_bound}, we have
\begin{align*}
&\PP\left\{
\exists t\in K:
\|Z_{\pi_\ell(t)}-Z_{t_0}\|
>
\sum_{j=1}^m(16u)^{\frac{1}{\alpha_j}}
d_j(\pi_\ell(t),t_0)
\right\}\le
2|\mathcal A_\ell|e^{-16u}\le
2e^{-((\log2)2^\ell-16u)}\le
2e^{-15u}.
\end{align*}

For every $n>\ell$, apply the increment condition with parameter
$16\cdot2^n$ to every edge in $E_n$. By
\eqref{eq:multi_regime_common_edge_cardinality} and another union
bound,
\begin{align*}
&\PP\left\{
\exists(s,r)\in E_n:\
\|Z_s-Z_r\|
>
\sum_{j=1}^m(16\cdot2^n)^{\frac{1}{\alpha_j}}d_j(s,r)
\right\}\le
2|E_n|e^{-16\cdot2^n}\le
2e^{-((16-\log2)2^n)}\le
2e^{-15\cdot2^n}.
\end{align*}
Since $2^{\ell+1}>u$, direct algebra gives
\begin{align*}
\sum_{n>\ell}2e^{-15\cdot2^n}\le
2\sum_{r\ge1}e^{-15u2^{r-1}}\le
2\sum_{r\ge1}e^{-15ur}=
\frac{2e^{-15u}}{1-e^{-15u}}.
\end{align*}
Because $u\ge2^{b_m}\ge2$, by some algebra, we have
\begin{align*}
2e^{-15u}
+
\frac{2e^{-15u}}{1-e^{-15u}}
\le
e^{-u}.
\end{align*}
Consequently, there exists an event $\mathcal E_u$ with
$\PP(\mathcal E_u)\ge1-e^{-u}$ on which all the preceding coarse and
fine increment inequalities hold simultaneously.

Fix $t\in K$ and we work conditioning on $\mathcal E_u$. Since
$\pi_n(t)=t$ for all sufficiently large $n$ and $Z_{t_0}=0$, we may
telescope and obtain
\begin{align*}
\|Z_t\|
&\le
\|Z_{\pi_\ell(t)}-Z_{t_0}\|
+
\sum_{n>\ell}
\|Z_{\pi_n(t)}-Z_{\pi_{n-1}(t)}\|\\
&\le
\sum_{j=1}^m(16u)^{1/\alpha_j}
d_j(\pi_\ell(t),t_0)+
\sum_{n>\ell}
\sum_{j=1}^m
(16\cdot2^n)^{1/\alpha_j}
d_j(\pi_n(t),\pi_{n-1}(t)).
\end{align*}

We bound the coarse term first. By the triangle inequality, we have
\begin{align*}
d_j(\pi_\ell(t),t_0)
\le
d_j(\pi_\ell(t),t)+v_j(t).
\end{align*}
When $u<2^{\ell+1}$, we have that $u^{1/\alpha_j}
<
2^{1/\alpha_j}2^{\ell/\alpha_j}
\le
a_{\boldsymbol\alpha}2^{\ell/\alpha_j}$.
Using \eqref{eq:multi_regime_common_representative_bound}, we get
\begin{align*}
\sum_{j=1}^m
(16u)^{1/\alpha_j}
d_j(t,\pi_\ell(t))\le
c_{\boldsymbol\alpha}^{(16)}a_{\boldsymbol\alpha}
\sum_{j=1}^m
2^{\ell/\alpha_j}d_j(t,\pi_\ell(t))\le
2c_{\boldsymbol\alpha}^{(16)}a_{\boldsymbol\alpha}R_\ell(t).
\end{align*}
The second part satisfies
\begin{align*}
\sum_{j=1}^m
(16u)^{1/\alpha_j}v_j(t)
\le
c_{\boldsymbol\alpha}^{(16)}
\sum_{j=1}^m u^{1/\alpha_j}v_j(t).
\end{align*}

For the fine increments, inequality
\eqref{eq:multi_regime_common_parent_edge_bound} gives
\begin{align*}
&\sum_{n>\ell}
\sum_{j=1}^m
(16\cdot2^n)^{1/\alpha_j}
d_j(\pi_n(t),\pi_{n-1}(t))\le
c_{\boldsymbol\alpha}^{(16)}
\sum_{n>\ell}
\cbr{
2R_n(t)+2a_{\boldsymbol\alpha}R_{n-1}(t)
}\le
2c_{\boldsymbol\alpha}^{(16)}
\rbr{1+a_{\boldsymbol\alpha}}
\sum_{n\ge b_m}R_n(t).
\end{align*}
Combining the coarse and fine estimates, and using
$R_\ell(t)\le\sum_{n\ge b_m}R_n(t)$, we obtain
\begin{align}
\|Z_t\|
&\le
c_{\boldsymbol\alpha}^{(16)}
\sum_{j=1}^m u^{1/\alpha_j}v_j(t)
+
2c_{\boldsymbol\alpha}^{(16)}
\rbr{1+2a_{\boldsymbol\alpha}}
\sum_{n\ge b_m}R_n(t).
\label{eq:multi_regime_common_chain_preconversion}
\end{align}

By the definition of $R_n(t)$ and the change of variables
$k=n-b_m$, we have the exact identity
\begin{align*}
\sum_{n\ge b_m}R_n(t)=&
\sum_{n\ge b_m}
\sum_{j=1}^m
2^{n/\alpha_j}d_j(t,T_{j,n-b_m})=
\sum_{j=1}^m
2^{b_m/\alpha_j}
\sum_{k\ge0}2^{k/\alpha_j}d_j(t,T_{j,k})=
\sum_{j=1}^m
2^{b_m/\alpha_j}G_j(t)\\
\le&
M_{m,\boldsymbol\alpha}
\sum_{j=1}^mG_j(t).
\end{align*}
In the last inequality, we use $M_{m,\boldsymbol\alpha}=
\max_{1\le j\le m}2^{b_m/\alpha_j}$. Inserting this identity into
\eqref{eq:multi_regime_common_chain_preconversion}, and observing that
the constant in \eqref{eq:multi_regime_explicit_finite_constant}
dominates both resulting coefficients, we prove
\eqref{eq:multi_regime_finite_chain_bound} when
$u\ge2^{b_m}$.

It remains to consider $1\le u<2^{b_m}$. Apply the result already
proved with $u_0:=2^{b_m}$.
Its failure probability satisfies
$e^{-u_0}
\le
e^{-u}$.
Moreover, since $u\ge1$,
\begin{align*}
u_0^{1/\alpha_j}
=
2^{b_m/\alpha_j}
\le
M_{m,\boldsymbol\alpha}u^{1/\alpha_j}.
\end{align*}
The high-$u$ estimate before the final absorption therefore yields
\begin{align*}
\|Z_t\|
&\le
c_{\boldsymbol\alpha}^{(16)}
M_{m,\boldsymbol\alpha}
\sum_{j=1}^m u^{1/\alpha_j}v_j(t)+
2c_{\boldsymbol\alpha}^{(16)}
\rbr{1+2a_{\boldsymbol\alpha}}
M_{m,\boldsymbol\alpha}
\sum_{j=1}^mG_j(t).
\end{align*}
This is bounded by the right-hand side of
\eqref{eq:multi_regime_finite_chain_bound} with the constant in
\eqref{eq:multi_regime_explicit_finite_constant}. Thus, we complete the proof.
\end{proof}

Combining the deterministic measure chains with the preceding
probabilistic lemma gives the finite ambient-measure estimate used in the
proof of the theorem.

\begin{lemma}
\label{lem:multi_regime_finite_ambient_pointwise}
There exists a finite constant
$B_{m,\boldsymbol\alpha}$, depending only on $m$ and
$\boldsymbol\alpha$, such that, for every finite set
$K\subseteq T$ containing $t_0$ and every $u\ge1$, with probability
at least $1-e^{-u}$, simultaneously for every $t\in K$,
\begin{align}
\|Z_t\|
\le
B_{m,\boldsymbol\alpha}
\sum_{j=1}^m
\cbr{
\Phi_j(t)+u^{1/\alpha_j}v_j(t)
}.
\label{eq:multi_regime_finite_ambient_bound}
\end{align}
\end{lemma}

\begin{proof}[Proof of Lemma
\ref{lem:multi_regime_finite_ambient_pointwise}]
By Lemma \ref{lem:multi_regime_pointwise_measure_chain}, for every
$j\in\cbr{1,\ldots,m}$, there exists an admissible sequence
$\mathcal T_j=(T_{j,k})_{k\ge0}$, anchored at $t_0$ and eventually
equal to $K$, such that, simultaneously for every $t\in K$,
\begin{align*}
G_j(t)
:=
\sum_{k\ge0}2^{k/\alpha_j}d_j(t,T_{j,k})
\le
C_{\alpha_j}^{(1)}
\rbr{\Phi_j(t)+v_j(t)}.
\end{align*}
Apply Lemma \ref{lem:multi_regime_finite_pointwise} to these $m$
sequences. Outside an event of probability at most $e^{-u}$, for every
$t\in K$, we have
\begin{align*}
\|Z_t\|\le
A_{m,\boldsymbol\alpha}
\left[
\sum_{j=1}^mG_j(t)
+
\sum_{j=1}^m u^{1/\alpha_j}v_j(t)
\right]\le
A_{m,\boldsymbol\alpha}
\sum_{j=1}^m
\left[
C_{\alpha_j}^{(1)}\Phi_j(t)
+
C_{\alpha_j}^{(1)}v_j(t)
+
u^{1/\alpha_j}v_j(t)
\right].
\end{align*}
Since $u\ge1$, we have $v_j(t)
\le
u^{1/\alpha_j}v_j(t)$. Therefore, we obtain
\begin{align*}
C_{\alpha_j}^{(1)}v_j(t)
+
u^{1/\alpha_j}v_j(t)
\le
\rbr{C_{\alpha_j}^{(1)}+1}
u^{1/\alpha_j}v_j(t).
\end{align*}
Consequently, \eqref{eq:multi_regime_finite_ambient_bound} holds with
\begin{align*}
B_{m,\boldsymbol\alpha}
:=
A_{m,\boldsymbol\alpha}
\max_{1\le j\le m}
\rbr{C_{\alpha_j}^{(1)}+1}.
\end{align*}
We finish the proof.
\end{proof}

We are now ready to pass from finite sets to the full separable index
space and prove the main theorem. The passage uses countable dense subsets
of rational upper-level sets. It does not allocate separate failure
probabilities to these sets.

\begin{proof}[Proof of Theorem \ref{thm:multi_regime_majorization}]
Fix $u\ge1$. For positive rational vectors
\begin{align*}
\boldsymbol q=(q_1,\ldots,q_m)\in\mathbb Q_{>0}^m,
\qquad
\boldsymbol s=(s_1,\ldots,s_m)\in\mathbb Q_{>0}^m,
\end{align*}
define
\begin{align*}
H_{\boldsymbol q,\boldsymbol s}
:=
\cbr{
t\in T:
\Phi_j(t)\le q_j
\text{ and }
v_j(t)\le s_j
\text{ for every }1\le j\le m
}.
\end{align*}
By Lemma \ref{lem:multi_regime_measurability}, every set
$H_{\boldsymbol q,\boldsymbol s}$ is Borel measurable.

We next choose countable dense subsets of these sets. Since $(T,\rho)$
is a separable pseudo-metric space, its metric quotient by the relation
$\rho(s,t)=0$ is a separable metric space and hence is second countable.
Every subspace of a second-countable space is second countable and
therefore separable. Consequently, each
$H_{\boldsymbol q,\boldsymbol s}$ is separable with respect to $\rho$.

For every pair $(\boldsymbol q,\boldsymbol s)$, choose a countable
$\rho$-dense subset
\begin{align*}
D_{\boldsymbol q,\boldsymbol s}
\subseteq
H_{\boldsymbol q,\boldsymbol s}.
\end{align*}
If the corresponding set is empty, take
$D_{\boldsymbol q,\boldsymbol s}=\varnothing$. The collection of
positive rational vectors is countable, and hence
\begin{align*}
D
:=
\cbr{t_0}
\cup
\bigcup_{
\boldsymbol q\in\mathbb Q_{>0}^m,\,
\boldsymbol s\in\mathbb Q_{>0}^m
}
D_{\boldsymbol q,\boldsymbol s}
\end{align*}
is countable. Beginning with the anchor, enumerate it as
$D=\cbr{t_0,t_1,t_2,\ldots}$. If $D$ is finite, repeat $t_0$ so that the same notation remains
valid. Define the increasing sequence of finite sets
\begin{align*}
K_N
:=
\cbr{t_0,t_1,\ldots,t_N},
\ N\ge1.
\end{align*}

For every $N\ge1$, let $O_N$ be the actual conclusion event
\begin{align*}
O_N
:=
\bigcap_{t\in K_N}
\cbr{
\|Z_t\|
\le
B_{m,\boldsymbol\alpha}
\sum_{j=1}^m
\cbr{
\Phi_j(t)+u^{1/\alpha_j}v_j(t)
}
}.
\end{align*}
By Lemma \ref{lem:multi_regime_finite_ambient_pointwise}, we have that
\begin{align*}
\PP(O_N)
\ge
1-e^{-u},
\ N\ge1.
\end{align*}
Moreover, since $K_N\subseteq K_{N+1}$ and the deterministic
right-hand side does not depend on $N$, the events are decreasing and $O_{N+1}
\subseteq
O_N$. 
By the continuity of probability,
\begin{align*}
\PP\left(
\bigcap_{N\ge1}O_N
\right)
=
\lim_{N\to\infty}\PP(O_N)
\ge
1-e^{-u}.
\end{align*}

Fix an outcome in $\bigcap_{N\ge1}O_N$. By the null-set modification
made at the beginning of the appendix, the sample path
$t\mapsto Z_t$ is $\rho$-continuous. Let $t\in T$ satisfy
\begin{align*}
\Phi_j(t)<\infty,
\ j\in\cbr{1,\ldots,m}.
\end{align*}
Every $v_j(t)$ is finite because the pseudo-metrics are finite-valued.
Choose arbitrary positive rational numbers such that
\begin{align*}
q_j>\Phi_j(t),
\ 
s_j>v_j(t),
\ j\in\cbr{1,\ldots,m}.
\end{align*}
Then $t\in H_{\boldsymbol q,\boldsymbol s}$. Since
$D_{\boldsymbol q,\boldsymbol s}$ is $\rho$-dense in this set, there
exists a sequence $(x_r)_{r\ge1}
\subseteq
D_{\boldsymbol q,\boldsymbol s}$ such that
\begin{align*}
\rho(x_r,t)
\rightarrow
0.
\end{align*}
Every $x_r$ belongs to the countable set $D$ and hence to some
$K_N$. On $\bigcap_{N\ge1}O_N$, we therefore have
\begin{align*}
\|Z_{x_r}\|
\le
B_{m,\boldsymbol\alpha}
\sum_{j=1}^m
\cbr{
\Phi_j(x_r)+u^{1/\alpha_j}v_j(x_r)
}\le
B_{m,\boldsymbol\alpha}
\sum_{j=1}^m
\cbr{
q_j+u^{1/\alpha_j}s_j
}.
\end{align*}
By the path continuity, we have
$Z_{x_r}
\rightarrow
Z_t
\ \text{in }(\mathbb B,\|\cdot\|)$. Therefore, we have,
\begin{align*}
\|Z_t\|
\le
B_{m,\boldsymbol\alpha}
\sum_{j=1}^m
\cbr{
q_j+u^{1/\alpha_j}s_j
}.
\end{align*}
Letting $q_j\downarrow\Phi_j(t)$ and
$s_j\downarrow v_j(t)$ through rational values for every $j$, we
obtain
\begin{align*}
\|Z_t\|
\le
B_{m,\boldsymbol\alpha}
\sum_{j=1}^m
\cbr{
\Phi_j(t)+u^{1/\alpha_j}v_j(t)
}.
\end{align*}
The event does not depend on $t$. Thus, the inequality holds
simultaneously for every point at which all the functionals are finite.
If at least one functional is infinite, the conclusion is trivial.

Finally, let $\delta\in(0,1)$ and set $u
:=
\log\frac{e}{\delta}$. Then $u\ge1$ and $e^{-u}
=
\frac{\delta}{e}
\le
\delta$.
Therefore, with probability at least $1-\delta$, simultaneously for
every $t\in T$,
\begin{align*}
\|Z_t\|
\le
B_{m,\boldsymbol\alpha}
\sum_{j=1}^m
\left[
\Phi_j(t)
+
v_j(t)
\left(
\log\frac{e}{\delta}
\right)^{1/\alpha_j}
\right].
\end{align*}
The theorem follows with
$C_{m,\boldsymbol\alpha}:=B_{m,\boldsymbol\alpha}$.
\end{proof}

\section{Proofs in Section \ref{sec:pointwise_ergodic_diffusion}}
\label{app:proofs_diffusion}

The diffusion results in Appendix \ref{app:mathematical_tools} are stated
directly in the notation of Section
\ref{sec:pointwise_ergodic_diffusion}. We verify the explicit Poincar\'e
constant, identify the long-run variance, and check the measurability and
continuity hypotheses of the mixed-tail theorem.

\subsection{The Poincar\'e constant and the long-run variance}

Since $b$ and $\sigma$ are globally Lipschitz, they are locally Lipschitz
and satisfy a linear-growth bound. Thus Assumption
\ref{ass:diffusion_coefficients} is a special case of Definition~5 and
equation~(4.26) of \citet{aeckerle2021concentration}. The statement
following that equation and their equation~(4.27) give the unique global
strong solution, its ergodicity, and its invariant density $\pi$.

We first verify the quantities in Lemma
\ref{lem:tool_loukianov_poincare}. Let
\begin{align*}
F(x):=
\int_{-\infty}^x\pi(y)dy,\ 
\overline F(x):=
\int_x^\infty\pi(y)dy,
\end{align*}
and define
\begin{align*}
B_+:=
\sup_{x>0}
\overline F(x)
\int_0^x\frac{2}{\ell(y)}dy,\ 
B_-:=
\sup_{x<0}
F(x)
\int_x^0\frac{2}{\ell(y)}dy.
\end{align*}
For $y\ge x\ge A$, by the drift assumption, we have that
\begin{equation}
G(y)-G(x)\le-2\gamma(y-x).
\label{eq:diffusion_G_right}
\end{equation}
Since $a\ge\underline\sigma^2$, it follows that, for $x\ge A$, we have
\begin{equation}\label{eq:diffusion_tail_ratio}
\overline F(x)=
\frac1{\mathsf Z}
\int_x^\infty\frac{e^{G(y)}}{a(y)}dy\le
\frac{e^{G(x)}}{2\gamma\mathsf Z\underline\sigma^2}=\frac{\ell(x)}{2\gamma\underline\sigma^2}.
\end{equation}
Moreover, we have
\begin{equation*}
\ell(x)
\int_0^x\frac{2}{\ell(y)}dy\le
\ell(x)
\int_0^A\frac{2}{\ell(y)}dy+2\int_A^x e^{-2\gamma(x-y)}dy\le
\ell(x)
\int_0^A\frac{2}{\ell(y)}dy
+
\frac1\gamma.
\end{equation*}
Hence $B_+<\infty$. Similarly, we have that $B_-<\infty$. According to Lemma \ref{lem:tool_loukianov_poincare}, we know that the stationary
diffusion is reversible and that the Poincare constant is
\begin{align}
C_P
:=
4\max\cbr{B_+,B_-}
<\infty
\label{eq:diffusion_explicit_CP}.
\end{align}

We next identify the long-run variance. First, we claim that 
$\cV_\pi(f)$ is a finite seminorm for every bounded Borel function $f$. Indeed, on
$[A,\infty)$, equation \eqref{eq:diffusion_tail_ratio} gives
\begin{align*}
\frac{R_f(x)^2}{\ell(x)}=\frac{(\int_{x}^\infty\overline{f}(y)\pi(y)dy)^2}{\ell(x)}\le\frac{4\|f\|_{\infty}^2(\int_{x}^{\infty}\pi(y)dy)^2}{\ell(x)}\le
4\|f\|_\infty^2
\frac{\overline F(x)^2}{\ell(x)}\le
\frac{2\|f\|_\infty^2}
{\gamma\underline\sigma^2}\overline F(x).
\end{align*}
Taking the two arguments in
\eqref{eq:diffusion_G_right} to be $A$ and $x$, we have that $G(x)
\le
G(A)-2\gamma(x-A),
\ x\ge A$.

Combining this inequality with
\eqref{eq:diffusion_tail_ratio}, we obtain $\overline F(x)
\le
\frac{e^{G(A)}}{2\gamma\mathsf Z\underline\sigma^2}
e^{-2\gamma(x-A)},\ x\ge A$.

Therefore, $\int_{A}^{\infty}\frac{2\|f\|_\infty^2}
{\gamma\underline\sigma^2}\overline F(x)dx$ is finite because of the exponential tail above. Similarly, the left
tail integral $\int_{-\infty}^{-A}\frac{2\|f\|_\infty^2}
{\gamma\underline\sigma^2}\overline F(x)dx$ is finite. Lastly, by nature, the integral $\int_{-A}^{A}\frac{2\|f\|_\infty^2}
{\gamma\underline\sigma^2}\overline F(x)dx$ is finite. Therefore, we obtain that
$\int_{\RR}\frac{R_f(x)^2}{\ell(x)}dx<\infty$ is finite and hence $\cV_\pi(f)
=
2\|\frac{R_f(x)}{\sqrt{\ell}(x)}\|_{L_2}<\infty$ is finite.

The centering map is linear, and hence we have that $R_{\lambda f+\eta h}
=
\lambda R_f+\eta R_h$ for all bounded Borel functions $f,h$ and all
$\lambda,\eta\in\RR$. Consequently,
\begin{align*}
\cV_\pi(\lambda f)
=
|\lambda|\cV_\pi(f),\ 
\cV_\pi(f+h)
\le
\cV_\pi(f)+\cV_\pi(h),
\end{align*}
where the second inequality follows from the triangle inequality in the
$L^2$ norm. Together with $\cV_\pi(0)=0$, this proves that
$\cV_\pi$ is a finite seminorm.

Now, to apply Lemma \ref{lem:tool_gao_bernstein}, it remains to show that $\cV_\pi(f)^2=\lim_{t\to\infty}
\frac1t
\operatorname{Var}_{\PP_\pi}
\left(
\int_0^t\overline f(X_s)ds
\right)$.
For arbitrary $t>0$, we write $Z_t(f)
:=
\frac1{\sqrt t}
\int_0^t\overline f(X_s)ds$. By Lemma \ref{lem:tool_vdvz_scalar_clt}, we have that
\begin{align}
Z_t(f)
\ \Longrightarrow\
\mathcal N\rbr{0,\cV_\pi(f)^2},
\ t\to\infty.
\label{eq:diffusion_scalar_clt}
\end{align}
We next upgrade this weak convergence to convergence of second moments.
Set
$A_f=
2\sqrt{C_P}\|\overline f\|_{L^2(\pi)},
B_f=
C_P\|\overline f\|_\infty$.
By Inequality \eqref{eq:tool_gao_coarse_two_sided} in Lemma \ref{lem:tool_gao_bernstein}, for $t\ge1$ and
$u>0$, we have that
\begin{align*}
\PP_\pi\cbr{
|Z_t(f)|>A_f\sqrt u+B_fu
}
\le2e^{-u}.
\end{align*}
Consequently, by direct algebra, for any $t$, we have
\begin{align*}
    \EE_\pi|Z_t(f)|^4=4\int_{0}^{\infty}x^3\PP_{\pi}(|Z_t(f)|>x)dx\le 8\int_{0}^{\infty}x^3[e^{-x^2/(4A_f^2)}+e^{-x/2B_f}]dx\le 64A_f^4+768B_f^4.
\end{align*}
Taking supremum over $t$, we obtain that
$\sup_{t\ge1}\EE_\pi|Z_t(f)|^4
\le
64A_f^4+768B_f^4
<\infty$. It follows that $\cbr{Z_t(f)^2:t\ge1}$ is uniformly integrable. Indeed,
for every $M>0$, by Markov inequality, we have
\begin{align*}
\sup_{t\ge1}
\EE_\pi\left[
Z_t(f)^2
\mathbf 1_{\{Z_t(f)^2>M\}}
\right]
\le
\frac1M
\sup_{t\ge1}\EE_\pi|Z_t(f)|^4
\longrightarrow0.
\end{align*}

By \eqref{eq:diffusion_scalar_clt} and continuous mapping theorem,
$Z_t(f)^2$ converges in distribution to $Y^2$, where
$Y\sim\mathcal N(0,\cV_\pi(f)^2)$. Uniform integrability therefore implies that $\EE_\pi Z_t(f)^2\rightarrow
\EE Y^2
=
\cV_\pi(f)^2$.

Finally, since $\pi(\overline f)=0$, we know that
$\EE_\pi Z_t(f)=0$, which implies $\EE_\pi Z_t(f)^2
=
\frac1t
\operatorname{Var}_{\PP_\pi}
\left(
\int_0^t\overline f(X_s)\,ds
\right)$. Therefore, we have that
\begin{align}
\lim_{t\to\infty}
\frac1t
\operatorname{Var}_{\PP_\pi}
\left(
\int_0^t\overline f(X_s)\,ds
\right)
=
\cV_\pi(f)^2.
\label{eq:diffusion_asymptotic_variance}
\end{align}
Thus, the assumptions in Lemma \ref{lem:tool_gao_bernstein} are satisfied. We apply Inequality \eqref{eq:tool_gao_variance_domination} and \eqref{eq:tool_gao_two_sided} to obtain
\begin{align}
\cV_\pi(f)^2
\le
2C_P\|\overline f\|_{L^2(\pi)}^2,
\label{eq:diffusion_variance_domination}
\end{align}

Now, we are ready to prove Proposition \ref{prop:diffusion_mixed_tail}
\begin{proof}[Proof of Proposition \ref{prop:diffusion_mixed_tail}]
The centered uniform norm and $\cV_\pi$ are seminorms, so $d_1$ and
$d_2$ are finite-valued pseudo-metrics and they are joint Borel measurable.

For $g:=(f-h)-\pi(f-h)$, equation
\eqref{eq:diffusion_variance_domination} gives
\begin{align*}
d_2(f,h)^2=
2\cV_\pi(f-h)^2\le
4C_P\|g\|_{L^2(\pi)}^2
\le
4C_P\|g\|_\infty^2.
\end{align*}
Thus, we have $d_2(f,h)
\le
2\sqrt{\frac{\tau}{C_P}}d_1(f,h)$ and every countable $d_1$-dense set is also dense for $d_1+d_2$.

The jointly Borel evaluation map implies that $f\mapsto\pi(f)$ is
Borel. Since the continuous adapted diffusion process $X$ is measurable, the map
\begin{align*}
(\omega,s,f)
\longmapsto
f(X_s(\omega))-\pi(f)
\end{align*}
is also jointly measurable. Parametrized integration therefore shows that
$(\omega,f)\mapsto Z_\tau(f)(\omega)$ is jointly measurable. For every
sample path, notice that
$|Z_\tau(f)-Z_\tau(h)|\le
\sqrt\tau\|g\|_\infty
=
\frac{\tau}{C_P}d_1(f,h)$,
so the sample paths are $(d_1+d_2)$-continuous. The zero function is an
anchor because $Z_\tau(0)=0$.

Recall that from the deduction above, the conditions in Lemma \ref{lem:tool_gao_bernstein} are satisfied. Hence, we apply Lemma \ref{lem:tool_gao_bernstein} to get that for every $u>0$,
\begin{align*}
\PP_\pi\cbr{
|Z_\tau(f)-Z_\tau(h)|
>
\sqrt{2u}\,\cV_\pi(f-h)
+
\frac{C_P\|g\|_\infty}{\sqrt\tau}u
}
\le2e^{-u}.
\end{align*}
This is exactly our result and we finish the proof.
\end{proof}

\begin{proof}[Proof of Corollary
\ref{thm:diffusion_pointwise_majorization}]
Proposition \ref{prop:diffusion_mixed_tail} verifies the hypotheses of
Theorem \ref{thm:mixed_tail_majorization} with
\begin{align*}
(\alpha_1,d_1,\mu_1)
&=(1,d_1,\mu_1),
&
(\alpha_2,d_2,\mu_2)
&=(2,d_2,\mu_2).
\end{align*}
The anchor radii and local ball-mass terms in that theorem are precisely
$v_1,v_2$ and $\Phi_1,\Phi_2$. Substitution gives
\eqref{eq:diffusion_pointwise_bound} on one event of probability at least
$1-\delta$, simultaneously for every $f\in\cF$.
\end{proof}

\section{Proofs in Section \ref{sec:higher_order_gaussian_chaos}}
\label{app:proofs_higher_order_gaussian_chaos}

\begin{proof}[Proof of Proposition
\ref{prop:higher_order_gaussian_chaos_multi_regime}]
We first prove the multi-regime increment inequality. Fix $s,t\in T$
and define the deterministic difference tensor
\begin{align*}
D_{s,t}
:=
A_s-A_t
=
\rbr{
a_{\boldsymbol i}(s)-a_{\boldsymbol i}(t)
}_{\boldsymbol i\in\mathcal I}.
\end{align*}
By Definition \ref{def:higher_order_gaussian_chaos}, we have
\begin{align}
C_s^{(q)}-C_t^{(q)}
=
\sum_{\boldsymbol i\in\mathcal I}
\rbr{
a_{\boldsymbol i}(s)-a_{\boldsymbol i}(t)
}
\prod_{r=1}^qg^{(r)}_{i_r}.
\label{eq:chaos_increment_representation}
\end{align}

Let $u\ge0$. If $0\le u\le\log2$, then we know that $2e^{-u}
\ge
1$ and \eqref{eq:higher_order_gaussian_chaos_increment} follows
immediately because every probability is at most one.

Suppose now that $u>\log2$, and set $p
:=
\max\cbr{2,u}$. Then $p\ge2$ and $p\ge u$. Moreover, we have
\begin{align}
p\le
4u.
\label{eq:chaos_moment_parameter_comparison}
\end{align}
In fact, if $u\ge2$, then $p=u$. If $\log2<u<2$, then
$p=2<4u$. Thus, \eqref{eq:chaos_moment_parameter_comparison} holds
in either case.

Apply Lemma \ref{lem:tool_latala_gaussian_chaos_moments} to the
tensor $D_{s,t}$. Using
\eqref{eq:chaos_moment_parameter_comparison}, we obtain that
\begin{align*}
\|C_s^{(q)}-C_t^{(q)}\|_{L^p}&\le
L_q
\sum_{\mathcal P\in\mathfrak P_q}
p^{|\mathcal P|/2}\|D_{s,t}\|_{\mathcal P}\le
L_q
\sum_{\mathcal P\in\mathfrak P_q}
(4u)^{|\mathcal P|/2}\|D_{s,t}\|_{\mathcal P}\\
&=
L_q
\sum_{\mathcal P\in\mathfrak P_q}
2^{|\mathcal P|}u^{|\mathcal P|/2}
\|D_{s,t}\|_{\mathcal P}\le
2^qL_q
\sum_{\mathcal P\in\mathfrak P_q}
u^{|\mathcal P|/2}
\|D_{s,t}\|_{\mathcal P},
\end{align*}
where the last inequality uses
$1\le|\mathcal P|\le q$.

Recall that
$K_q=e2^qL_q$. By the definition of $d_{\mathcal P}$, the preceding
inequality implies that
\begin{align}
\|C_s^{(q)}-C_t^{(q)}\|_{L^p}
\le
\frac1e
\sum_{\mathcal P\in\mathfrak P_q}
u^{|\mathcal P|/2}d_{\mathcal P}(s,t).
\label{eq:chaos_increment_lp_control}
\end{align}
If the left-hand side is zero, then
$C_s^{(q)}-C_t^{(q)}=0$ almost surely and the desired inequality is
immediate. Otherwise, Markov's inequality and
\eqref{eq:chaos_increment_lp_control} give
\begin{align*}
&\PP\cbr{
\left|C_s^{(q)}-C_t^{(q)}\right|
>
\sum_{\mathcal P\in\mathfrak P_q}
u^{|\mathcal P|/2}d_{\mathcal P}(s,t)
}\le
\PP\cbr{
\left|C_s^{(q)}-C_t^{(q)}\right|
>
e\|C_s^{(q)}-C_t^{(q)}\|_{L^p}
}\le
e^{-p}
\le
e^{-u}.
\end{align*}
Together with the case $0\le u\le\log2$, this proves
\eqref{eq:higher_order_gaussian_chaos_increment} for every $u\ge0$.

We now verify the pseudo-metrics satisfy the assumptions in Theorem \ref{thm:multi_regime_majorization}. Fix
\begin{align*}
\mathcal P
=
\cbr{I_1,\ldots,I_k}
\in
\mathfrak P_q.
\end{align*}
For feasible arrays $x^{(1)},\ldots,x^{(k)}$ in
\eqref{eq:gaussian_chaos_partition_norm}, define $y_{\boldsymbol i}
:=
\prod_{r=1}^k
x^{(r)}_{\boldsymbol i_{I_r}}$.

Since the blocks $I_1,\ldots,I_k$ form a partition of $[q]$, we have that $\sum_{\boldsymbol i\in\mathcal I}y_{\boldsymbol i}^2=
\prod_{r=1}^k
\sum_{\boldsymbol i_{I_r}}
\rbr{x^{(r)}_{\boldsymbol i_{I_r}}}^2
\le
1$.

Therefore, by the Cauchy-Schwarz inequality, we have
\begin{align}
\|A\|_{\mathcal P}
\le
\|A\|_{\mathrm F},
\label{eq:partition_norm_frobenius_domination}
\end{align}
where $\|A\|_{\mathrm F}
:=
\left(
\sum_{\boldsymbol i\in\mathcal I}a_{\boldsymbol i}^2
\right)^{1/2}$.
In particular, $\|\cdot\|_{\mathcal P}$ is finite-valued and we know that $\|\cdot\|_{\mathcal P}$ is a norm.

Therefore, it follows that
\begin{align*}
d_{\mathcal P}(s,t)
=
K_q\|A_s-A_t\|_{\mathcal P}
\end{align*}
is a pseudo-metric on $T$. Since $t\mapsto A_t$ is Borel measurable,
the map $(s,t)\mapsto
A_s-A_t$ is Borel measurable from $T\times T$ to the finite-dimensional tensor
space. Every norm on a finite-dimensional vector space is continuous.
Consequently, $d_{\mathcal P}$ is jointly Borel measurable.

We now prove separability. Let $\mathcal P_0
:=
\cbr{[q]}$ be the one-block partition. By definition, we know that $\|A\|_{\mathcal P_0}
=
\|A\|_{\mathrm F}$.

Combining this identity with
\eqref{eq:partition_norm_frobenius_domination}, for every $s,t\in T$
we obtain
\begin{align}
K_q\|A_s-A_t\|_{\mathrm F}
\le
\rho(s,t)
\le
K_q\mathsf B_q\|A_s-A_t\|_{\mathrm F}.
\label{eq:chaos_rho_frobenius_equivalence}
\end{align}
Therefore, $\|\cdot\|_F$ and $\rho$ are equivalent metrics. Notice that $\mathcal M
:=
\cbr{A_t:t\in T}$ is a subset of a finite-dimensional Euclidean space and is therefore
separable under the Frobenius norm. Thus, it is also separable under $\rho$.

It remains to prove joint measurability and sample continuity. Since
$t\mapsto a_{\boldsymbol i}(t)$ is Borel measurable for every
$\boldsymbol i\in\mathcal I$, the finite-sum representation
\begin{align*}
C_t^{(q)}
=
\sum_{\boldsymbol i\in\mathcal I}
a_{\boldsymbol i}(t)
\prod_{r=1}^qg^{(r)}_{i_r}
\end{align*}
shows that the map
$(\omega,t)\mapsto C_t^{(q)}(\omega)$ is jointly measurable.

For every $s,t\in T$, another application of the Cauchy--Schwarz
inequality gives
\begin{align*}
\left|C_s^{(q)}-C_t^{(q)}\right|&=
\left|
\sum_{\boldsymbol i\in\mathcal I}
\rbr{a_{\boldsymbol i}(s)-a_{\boldsymbol i}(t)}
\prod_{r=1}^qg^{(r)}_{i_r}
\right|\le
\|A_s-A_t\|_{\mathrm F}
\left(
\sum_{\boldsymbol i\in\mathcal I}
\prod_{r=1}^q\rbr{g^{(r)}_{i_r}}^2
\right)^{1/2}\\
&=
\|A_s-A_t\|_{\mathrm F}
\prod_{r=1}^q\|g^{(r)}\|_2\le
\frac1{K_q}
\prod_{r=1}^q\|g^{(r)}\|_2\,
\rho(s,t),
\end{align*}
where the last inequality follows from the lower bound in
\eqref{eq:chaos_rho_frobenius_equivalence}. The  coefficient
$\frac1{K_q}
\prod_{r=1}^q\|g^{(r)}\|_2$ is finite almost surely. Therefore, with probability one, the map
$t\mapsto C_t^{(q)}$ is $\rho$-Lipschitz and hence
$\rho$-continuous.

Finally, $A_{t_0}=0$ implies $C_{t_0}^{(q)}=0$ almost surely. 

All the assertions of the proposition are proved.
\end{proof}

\begin{proof}[Proof of Theorem \ref{thm:higher_order_gaussian_chaos_pointwise_envelope}]
By Proposition \ref{prop:higher_order_gaussian_chaos_multi_regime}, we know that the decoupled Gaussian chaos satisfies the conditions in Theorem \ref{thm:multi_regime_majorization}. Thus, we apply the theorem directly and obtain the conclusion.
\end{proof}

\begin{proof}[Proof of Corollary \ref{cor:quadratic_gaussian_chaos_pointwise_envelope}]
    We set $q=2$ in Theorem \ref{thm:higher_order_gaussian_chaos_pointwise_envelope} and by direct algebra, the corollary is proved.
\end{proof}

\bibliographystyle{plainnat}
\bibliography{refs}

\end{document}